\documentclass[11pt]{amsart}

\usepackage{amsmath}
\usepackage{stmaryrd}
\usepackage{amsfonts}
\usepackage{amssymb}
\usepackage{latexsym}
\usepackage{amscd}
\usepackage{mathrsfs}
\usepackage[all]{xy}

\usepackage{esint}

\usepackage{}

\newdimen\AAdi%
\newbox\AAbo%
\def\AAk#1#2{\s_etbox\AAbo=\hbox{#2}\AAdi=\wd\AAbo\kern#1\AAdi{}}%
\def\AAr#1#2#3{\s_etbox\AAbo=\hbox{#2}\AAdi=\ht\AAbo\raise#1\AAdi\hbox{#3}}%
\font\tenmsb=msbm10 at 12pt \font\sevenmsb=msbm7 at 8pt
\font\fivemsb=msbm5 at 6pt
\newfam\msbfam
\textfont\msbfam=\tenmsb \scriptfont\msbfam=\sevenmsb
\scriptscriptfont\msbfam=\fivemsb

\newtheorem{theorem}{Theorem}

\newtheorem{remark}[theorem]{Remark}

\newtheorem{corollary}[theorem]{Corollary}

\newtheorem{lemma}[theorem]{Lemma}
\newtheorem{proposition}[theorem]{Proposition}

\numberwithin{equation}{section} \numberwithin{theorem}{section}

\renewcommand{\topmargin}{0cm}
\renewcommand{\oddsidemargin}{5mm}
\renewcommand{\evensidemargin}{5mm}
\renewcommand{\textwidth}{150mm}
\renewcommand{\textheight}{230mm}

\def\R{\mathbb R}

\def\na{\nabla}

\def\f#1#2{\frac{#1}{#2}}

\def\a{\alpha}
\def\be{\beta}

\def\r{\Re_{I\!V}}

\def\p#1{\partial #1}

\def\de{\delta}
\def\De{\Delta}
\def\e{\eta}
\def\ep{\epsilon}

\def\G{\Gamma}
\def\g{\gamma}
\def\k{\kappa}
\def\la{\lambda}
\def\La{\Lambda}
\def\lan{\langle}
\def\ran{\rangle}

\def\Om{\Omega}
\def\th{\theta}
\def\Th{\Theta}
\def\si{\sigma}
\def\Si{\Sigma}

\def\r{\rho}
\def\z{\zeta}

\begin{document}

\title
[Minimal hypersurfaces in manifolds]
{Minimal hypersurfaces in manifolds of Ricci curvature bounded below}

\author{Qi Ding}
\address{Shanghai Center for Mathematical Sciences, Fudan University, Shanghai 200438, China}

\email{dingqi@fudan.edu.cn}
\thanks{
The author is partially supported by NSFC 11871156 and NSFC 11922106.}

\begin{abstract}
In this paper, we study the angle estimate of distance functions from minimal hypersurfaces in manifolds of Ricci curvature bounded from below using Colding's method in \cite{C}.
With Cheeger-Colding theory, we obtain the Laplacian comparison for limits of distance functions from minimal hypersurfaces in the version of Ricci limit space.
As an application, if a sequence of minimal hypersurfaces converges to a metric cone $CY\times\R^{n-k}(2\le k\le n)$ in a non-collapsing metric cone $CX\times\R^{n-k}$ obtained from ambient manifolds of almost nonnegative Ricci curvature, then we can prove a Frankel property for the cross section $Y$ of $CY$. 
Namely, $Y$ has only one connected component in $X$.

\end{abstract}

\maketitle

\section{Introduction}

Let $B_{2R}(p_i)$ be a sequence of $(n+1)$-dimensional smooth geodesic balls with $Ric\ge-n\k^2$ for some constant $\k\ge0$ such that $B_{2R}(p_i)$ converges to a metric ball $\overline{B_{2R}(p_\infty)}$ in the Gromov-Hausdorff sense.
Then for each integer $i\ge1$, there is an $\ep_i$-Gromov-Hausdorff approximation $\Phi_i:\, \overline{B_{2R}(p_i)}\to\overline{B_{2R}(p_\infty)}$ for some sequence $\ep_i\to0$.
Let $V_i$ be a sequence of rectifiable stationary $n$-varifold in $B_{2R}(p_i)$ such that spt$V_i\cap B_{R}(p_i)\neq\emptyset$. Denote $M_i\triangleq\mathrm{spt}V_i$.
Up to choosing the subsequence, we can assume that $\Phi_i(M_i)$ converges to $M_\infty\subset\overline{B_{2R}(p_\infty)}$ in the Hausdorff sense.
For studying $M_\infty$, we need that the renormalized volume of $M_i$ is uniformly bounded, i.e.,
\begin{equation}\aligned\label{BdMiBRipi0}
\limsup_{i\rightarrow\infty}\f{\mathcal{H}^n(M_i\cap B_{R}(p_i))}{\mathcal{H}^{n+1}(B_R(p_i))}<\infty.
\endaligned
\end{equation}
%The Euclidean version of \eqref{BdMiBRipi0} 
The condition of uniformly bounded volume
is always assumed for compactness of stationary varifolds in Euclidean space (see \cite{S} for instance).

In this paper, we will study $M_\infty\cap B_R(p_\infty)$ using the distance function $\r_{M_\infty}$ from $M_\infty$ on $B_{2R}(p_\infty)$.
Our object is twofold. On the one hand, we prove the connectivity of the cross section of minimal cones in a class of metric cones(see Theorem \ref{ConnectedY0}), which will be used to establish Poincar\'e inequality on minimal graphs over manifolds of Ricci curvature bounded below in \cite{D2}.
On the other hand, we wish to study minimal hypersurfaces in manifolds of Ricci curvature bounded below via understanding the local geometric structure of $M_\infty$.

In general, a $L^2$-Hessian estimate may not exist for the distance functions from $M_i$ (compared to distance functions from fixed points in manifolds).
However, Colding \cite{C} can approach distance functions from points successfully using harmonic functions.
The harmonic functions admit $L^2$-Hessian estimates using the Bochner formula and the Cheng-Yau gradient estimate \cite{CY}.
After ingenious partitions twice, Colding got the angle estimate for the distance functions in an integral version.
Based on Colding's method \cite{C}, we can deduce a similar angle estimate for the distance function from $M_i$ for each $i$ (see Lemma \ref{bzth*}).

In \cite{CCo1,CCo2,CCo3}, Cheeger-Colding established the structure theory of the Ricci limit space $B_{2R}(p_\infty)$.
Let $\nu$ denote the renormalized limit measure from $B_R(p_i)$ (see \eqref{nuinfty}). 
Cheeger-Colding proved the rectifiability of $B_{2R}(p_\infty)$, and that $B_{2R}(p_\infty)$ admits cotangent bundle. This enables us to study derivative and Laplacian for Lipschitz functions on $B_{2R}(p_\infty)$ (see also Cheeger \cite{Ch}). We will give a brief introduction for this knowledge in $\S 2$.
In \cite{H1}, Honda studied convergence of the differentials of Lipschitz functions via radial derivatives of distance functions with respect to the measured Gromov-Hausdorff topology.
Combining Honda's results and Lemma \ref{bzth*}, we can prove the following convergence (see Theorem \ref{c1}).
\begin{theorem}\label{DerMiconv000}
Let $B_{2R}(p_i)$, $B_{2R}(p_\infty)$, $M_i$, $M_\infty$ as in the first paragraph with \eqref{BdMiBRipi0}. 
For each Lipschitz function $\phi$ on $B_R(p_\infty)$, there is a sequence of Lipschitz functions $\phi_i$ on $B_R(p_i)$ satisfying $(\phi_i,d\phi_i)\to(\phi,d\phi)$ on $B_R(p_\infty)$ and $\limsup_{i\rightarrow\infty}\mathbf{Lip}\,\phi_i\le\mathbf{Lip}\,\phi$ such that
\begin{equation}\aligned
\lim_{i\rightarrow\infty}\fint_{B_R(p_{i})}\left\lan d\r_{M_{i}},d\phi_{i}\right\ran=\fint_{B_R(p_\infty)}\left\lan d\r_{M_\infty},d\phi\right\ran d\nu.
\endaligned
\end{equation}
Furthermore, if $\phi$ has compact support in $B_R(p_\infty)\setminus M_\infty$, then we can require that the function $\phi_i$ has compact support in $B_R(p_i)\setminus M_i$.
\end{theorem}
From Theorem \ref{DerMiconv000} and the Laplacian comparison for $\r_{M_i}$ by Heintze-Karcher \cite{HeK} (see also Lemma 7.1 in \cite{D1} for the distribution sense),
we obtain (see \eqref{drMdphige0})
\begin{equation}\label{rMinfnk0}
\int_{B_R(p_\infty)}\left\lan d\r_{M_\infty},d\phi\right\ran d\nu\ge-n\k\int_{B_R(p_\infty)}\phi\tanh\left(\k\r_{M_\infty}\right)d\nu
\end{equation}
for each nonnegative Lipschitz function $\phi$ on $B_R(p_\infty)$ with compact support in $B_R(p_\infty)\setminus M_\infty$.

Let $R_i\ge0$ be a sequence with $R_i\rightarrow\infty$ as $i\rightarrow\infty$.
Let $B_{R_i}(q_i)$ be a sequence of $(n+1)$-dimensional smooth geodesic balls with Ricci curvature $\ge-nR_i^{-2}$ such that
$(B_{R_i}(q_i),q_i)$ converges to a metric cone $(\mathbf{C},\mathbf{o})$ in the pointed Gromov-Hausdorff sense with $\ep_i$-Gromov-Hausdorff approximations $\Phi_i$ for some $\ep_i\to0$.
Suppose $\liminf_{i\to\infty}\mathcal{H}^{n+1}(B_1(q_i))>0$, and $\mathbf{C}$ splits off a Euclidean factor $\R^{n-k}$ isometrically for some integer $1\le k\le n$.
Then there is a $k$-dimensional metric space $X$ such that $\mathbf{C}=CX\times\R^{n-k}$, where $CX$ is a metric cone of the cross section $X$ with the vertex $o$.
For each integer $i\ge1$, let $M_i$ be the support of rectifiable stationary $n$-varifold in $B_{R_i}(q_i)$ with $q_i\in M_i$ and \eqref{BdMiBRipi0} for some $R\ge2$.
Suppose that $\Phi_i(M_i)$ converges in the Hausdorff sense to a metric cone $CY\times\R^{n-k}$, where $CY$ is a metric cone with the vertex at $o$ and the cross section $Y\subset X$.

From \eqref{rMinfnk0}, we can prove that the distance function from $Y$ on $X$ is (strongly) superharmonic in the weak sense.
Then with the mean value inequality for weak superharmonic functions on $X$, we can derive the following result.
\begin{theorem}\label{ConnectedY0}
For $k\ge2$, $Y$ is connected in $X$. Namely, there are no non-empty closed sets $Y_1,Y_2\subset X$ with $Y=Y_1\cup Y_2$ and $Y_1\cap Y_2=\emptyset$.
\end{theorem}
Here, $k\ge2$ is necessary in Theorem \ref{ConnectedY0} since there are stationary varifolds in $\R^2$ composed by several ($\ge3$) radial lines through the origin. From Theorem \ref{ConnectedY0}, clearly the cross section of $CY\times\R^l$ is connected in the cross section of $CX\times\R^l$ for any integer $1\le l\le n-k$.
A simple version of Theorem \ref{ConnectedY0} is a well-known result that every minimal hypercone in Euclidean space $\R^{n+1}$ has the connected cross section in the unit sphere $\mathbb{S}^n$.
Theorem \ref{ConnectedY0} can be seen as a Frankel property in the metric space $X$, 
where Frankel \cite{F} proved that if $\Si_1,\Si_2$ are two complete compact minimal hypersurfaces in a complete connected manifold of positive Ricci curvature, then $\Si_1,\Si_2$ must intersect.

The paper is organized as follows. 
In $\S 3$, we investigate the measure related to the Laplacian of distance functions from hypersurfaces with bounded mean curvature using functions of bounded variation. In $\S 4$, we study the angle estimate for such distance functions based on Colding's method \cite{C}. In $\S 5$, we study the limits of the Laplacian of distance functions from minimal hypersurfaces in a sequence of manifolds of Ricci curvature bounded below.
As an application, in $\S 6$ we prove a Frankel property on cross sections of a class of metric cones.

\section{Preliminary}

Let $(X,d)$ be a complete metric space with the Radon measure $\mu$. The measure $\mu$ is said to be \emph{Ahlfors $k$-regular} at $x\in X$,
if there exists a constant $K_x\ge1$ such that
$$K_x^{-1}r^k\le\mu(B_r(x))\le K_xr^k$$
for all $0<r\le 1$.
The space $X$ is said to be \emph{$\mu$-rectifiable} (see Cheeger-Colding \cite{CCo3} for instance), if there exists an
integer $m$, a countable collection of Borel subsets, $C_{k,i}\subset X$ with $k\le m$ and bi-Lipschitz maps $\phi_{k,i}:\, C_{k,i}\rightarrow\phi_{k,i}(C_{k,i})\subset\R^k$ such that
\begin{itemize}
  \item [i)] $\mu\left(X\setminus\cup_{k,i}C_{k,i}\right)=0$;
  \item [ii)] $\mu$ is Ahlfors $k$-regular at all $x\in C_{k,i}$.
\end{itemize}

Let $f$ be a Borel function on $X$, and $g:\, X\rightarrow[0,\infty]$ be a Borel function such that for all points
$x_1,x_2\in X$ and all continuous rectifiable curves $c:\, [0,l]\rightarrow X$, parameterized by arclength $s$ with $c(0)=x_1$, $c(l)=x_2$, we have
\begin{equation}\aligned\label{fLipgradg}
|f(x_2)-f(x_1)|\le\int_0^l g(c(s))ds.
\endaligned
\end{equation}
In this case, $g$ is called an \emph{upper gradient} for $f$ (see \cite{Ch,CCo3}).
For a Lipschitz function $f$ on $X$, we define the pointwise Lipschitz constant
\begin{equation}\aligned\label{DefLip}
\mathrm{Lip}\, f(x)=\limsup_{r\rightarrow0}\sup_{y\in\p B_r(x)}\f{|f(x)-f(y)|}{r}=\limsup_{y\rightarrow x, y\neq x}\f{|f(x)-f(y)|}{d(x,y)}
\endaligned
\end{equation}
for each $x\in X$. Clearly, \eqref{fLipgradg} holds with $g=\mathrm{Lip}\, f$ (see Cheeger \cite{Ch} for further discussion).
Let $\mathbf{Lip}\, f=\sup_{x\in X}\mathrm{Lip}\, f(x)$ denote the Lipschitz constant of $f$ on $X$.

For $R>0$ and $\k\ge0$,
let $B_{2R}(p_i)$ be a sequence of $(n+1)$-dimensional smooth geodesic balls with $Ric\ge-n\k^2$ such that $B_{R}(p_i)$ converges to an $(n+1)$-dimensional metric ball $\overline{B_{R}(p)}$ in the Gromov-Hausdorff sense.
From Cheeger-Colding \cite{CCo1,CCo2,CCo3}, there is a unique Radon measure $\nu$ on $B_{R}(p)$ given by
\begin{equation}\aligned\label{nuinfty}
\nu(B_r(y))=\lim_{i\rightarrow\infty}\f{\mathcal{H}^{n+1}(B_r(y_i))}{\mathcal{H}^{n+1}(B_R(p_i))}
\endaligned
\end{equation}
for every ball $B_r(y)\subset B_{R}(p)$ and for each sequence $y_i\in B_{R}(p_i)$ converging to $y$.
Here, $\mathcal{H}^{n+1}$ is the $(n+1)$-dimensional Hausdorff measure.
Denote $Z=B_{R}(p)$.
From \cite{CC,CCo3}, there holds the segment inequality (a refinement of the Poincar$\mathrm{\acute{e}}$ inequality) for $B_R(p_i)$ as well as the limit $Z$.
Moreover, Theorem 5.5 and Theorem 5.7 in \cite{CCo3} tell us that $Z$ is $\nu$-rectifiable, and the following condition iii) holds:
\begin{itemize}
  \item [iii)] For all $x\in \cup_{k,i}C_{k,i}$ and all $\la>0$, there exists $C_{k,i}$ such that $x\in C_{k,i}$ and the map $\phi_{k,i}:\, C_{k,i}\rightarrow\phi_{k,i}(C_{k,i})\subset\R^k$ is $e^{\pm\la}$-bi-Lipschitz.
\end{itemize}
A point $z\in Z$ is said to be a $k$-regular point if every tangent cone at $z$ is isometric to $\R^k$.
Let $\mathcal{R}_k$ denote the set of $k$-regular points in $Z$.
By Theorem 1.18 of Colding-Naber \cite{CN}, there is a unique integer $k_*\le n+1$ such that $\nu(Z\setminus \mathcal{R}_{k_*})=0$.

From \cite{Ch,CCo3}, a finite dimensional cotangent ($L_\infty$ vector) bundle $T^*Z$ exists on $Z$, and for each Lipschitz function $f$ on an open set $U\subset Z$ one can define \emph{differential} of $f$, denoted by $df$, which is a $\nu$-a.e. well defined $L_\infty$ section of $T^*U\subset T^*Z$, and the pointwise inner product $\lan df_1,df_2\ran$ $\nu$-a.e. for every Lipschitz functions $f_1,f_2$ on $U$.
Moreover, there is a Borel set $U^*\subset U$ such that $\nu(U\setminus U^*)=0$, and $|df_i|(x)=\mathrm{Lip}\, f_i(x)$ for every $x\in U^*$ and $i=1,2$. Obviously, the differential here satisfies the Leibnitz rule. Moreover,
\begin{equation}\aligned\label{quad}
\lan df_1,df_2\ran=\f14\mathrm{Lip}\,(f_1+f_2)^2-\f14\mathrm{Lip}\,(f_1-f_2)^2\qquad \mathrm{on}\ U^*.
\endaligned
\end{equation}
As a consequence, the Dirichlet energy $|\mathrm{Lip}\, f|$ is associated to a quadratic form which is obtained from the bilinear form
$$\int_U\lan df_1,df_2\ran d\nu.$$
From the uniqueness of strong derivatives in $L^2$ sense (Theorem 6.7 in \cite{CCo3}), it follows that the corresponding
Laplace operator is linear self-adjoint.
With (2.5) in \cite{CCo3}, H\"older inequality and Theorem 1 in \cite{HaK}, the following Poincar$\mathrm{\acute{e}}$ inequality
\begin{equation}\aligned\label{qPoincare}
\int_{B_r(z)}\left|f-\fint_{B_r(z)}fd\nu\right|^qd\nu\le c_{n,\k r }r^q\int_{B_{r}(z)}|df|^qd\nu
\endaligned
\end{equation}
holds for any $B_{2r}(z)\subset Z$, $q\ge1$, any Lipschitz function $f$ on $B_{2r}(z)$, where $c_{n,\k r}$ is a constant depending only on $n,\k r$, and $\fint_{B_r(z)}fd\nu$ is the average of $f$ on $B_r(z)$ defined by $\f1{\nu(B_r(z))}\int_{B_r(z)}fd\nu$. Moreover, the inequality \eqref{qPoincare} holds for all functions in the Sobolev space $H^{1,q}(Z)$ (see $\S 4$ in \cite{Ch}).

Let $N$ be an $(n+1)$-dimensional complete Riemannian manifold.
For an open set $\Om\subset N$,
there is an integer $m\ge1$ so that $\overline{\Om}$ is properly embedded in $\R^{n+m}$.
We use the definitions of Allard \cite{Al} for varifolds on $N$ (see also chapter 8 of \cite{S}, or \cite{wn1}).
Let $G_{n,m}$ denote the (Grassmann) manifold including all the $n$-dimensional subspace of $\R^{n+m}$.
An $n$-varifold $V$ in $\Om$ is a Radon measure on
$$G_{n,m}(\Om)=\{(x,T)|\, x\in \Om,\, T\in G_{n,m}\cap T_xN\}.$$
An $n$-rectifiable varifold in $\Om$ is an $n$-varifold in $\Om$ with support on countably $n$-rectifiable sets.
The $n$-rectifiable varifold $V$ is said to have the generalized mean curvature (vector field) $H$ in $\Om$ if
\begin{equation}\aligned\nonumber
\int_{G_{n,m}(\Om)}\mathrm{div}_\omega YdV(x,\omega)=-\int_{G_{n,m}(\Om)}\lan Y,H\ran dV(x,\omega)
\endaligned
\end{equation}
for each $Y\in C^\infty_c(\Om,\R^{n+m})$ with $Y(x)\in T_xN$ for each $x\in \Om$. In particular, $H(x)\in T_xN$ for a.e. $x\in \mathrm{spt} V$. The notion of $n$-rectifiable varifold obviously generalises the notion of $n$-dimensional $C^1$-submanifold.

$\mathbf{Notional\ convention}.$ If a point $q$ belongs to some metric space $Z$ defined in the paper, we let $d(q,\cdot)$ denote the distance function on $Z$ from $q$, and always let $B_r(q)$ denote the geodesic ball in this $Z$ with radius $r$ and centered at $q$. 
For a subset $K\subset Z$, let $B_r(K)$ denote $r$-tubular neighborhood of $K$ in $Z$. 
For each integer $k>0$, let $\omega_k$ denote the volume of $k$-dimensional unit Euclidean ball, and $\mathcal{H}^k$ denote the $k$-dimensional Hausdorff measure.
When we write an integation on a subset of a Riemannian manifold w.r.t. some volume element, we always omit the volume element if it is associated with the standard metric of the given manifold.

\section{Laplacian of distance functions from hypersurfaces}

Let $N$ be an $(n+1)$-dimensional complete Riemannian manifold with Ricci curvature $\mathrm{Ric}_N\ge-n\k^2$ on $B_{R+R'}(p)$ for some constants $\k\ge0$ and $R'\ge R>0$.
Let $\De_N$ denote the Laplacian of $N$, and $\na$ be the Levi-Civita connection of $N$.
Let $V$ be an $n$-rectifiable varifold in $B_{R+R'}(p)$ with the generalized mean curvature $H$ in $B_{R+R'}(p)$ (see its definition in $\S 2$). 
Let  $\La_0$ be a positive constant, and $\Si$ be the support of $V$. Suppose 
$$H_0\triangleq\sup_\Si|H|<\infty,\quad\Si\cap B_{R'}(p)\neq\emptyset, \quad \mathcal{H}^n(\Si)\le\La_0.$$ 
By a similar argument of Theorem 17.7 in \cite{S} (see also the proof of Theorem 5.1 in \cite{DJX}), there are constants $r_0>0,\be_0>\omega_n/2$ depending on $H_0$, $\La_0$ and the geometry of $B_{R+R'}(p)$ such that
\begin{equation}\aligned\label{locVolSi}
\f12\omega_n r^n\le\mathcal{H}^n(B_r(x)\cap \Si)\le\be_0r^n
\endaligned
\end{equation}
for any $x\in \Si\cap B_R(p$) and any $r\in(0,r_0)$. 

Let $\r_\Si$ denote the distance function from $\Si$ in $N$.
At a differentiable point $x\in B_{R}(p)\setminus\Si$ of $\r_\Si$,
there exist a unique $x'\in \Si$ and a unique non-zero vector $v_x\in\R^{n+1}$ with $|v_x|=\r_\Si(x)$ such that
$\mathrm{exp}_{x'}(v_x)=x$.
Let $\g_x$ denote the geodesic $\mathrm{exp}_{x'}(tv_x/|v_x|)$ from $t=0$ to $t=|v_x|$. In particular, $\r_\Si$ is smooth at $\g_x(t)$ for each $t\in(0,|v_x|]$.
Let $H_x(t), A_x(t)$ denote the mean curvature function (pointing out of $\{\r_\Si<t\}$), the second fundamental form of the level set $\{\r_\Si=t\}$ at $\g_x(t)$, respectively.
From Heintze-Karcher \cite{HeK}, there is a function $\Th$ with $|\Th|\le\max\{H_0,n\k\}$ such that 
\begin{equation}\aligned\label{DeNSiTh000}
\De_N\r_\Si=-H_x(t)\big|_{t=\r_\Si}\le \Th\qquad \mathrm{at}\ \g_x(\r_\Si).
\endaligned
\end{equation}
In fact, from the variational argument,
\begin{equation}\label{HtRic}
\f{\p H_x}{\p t}=|A_x|^2+Ric\left(\dot{\g}_x,\dot{\g}_x\right)\ge\f1n|H_x|^2-n\k^2.
\end{equation}
Clearly, the above ordinary differential inequality implies $H_x(t)\ge-\max\{H_0,n\k\}$ for each $t\in[0,|v_x|]$. So one can choose the function $\Th$ satisfying $|\Th|\le\max\{H_0,n\k\}$. (For a special situation, one may choose a suitable $\Th$, where the argument in $\S 3$ and $\S 4$ still works.)
From \eqref{DeNSiTh000} and the argument of the proof of Lemma 7.1 in \cite{D1},
\begin{equation}\aligned\label{DeNSiTh}
\De_N\r_\Si\le \Th\qquad \mathrm{on}\ B_{R}(p)\setminus \Si
\endaligned
\end{equation}
in the distribution sense. Namely,
\begin{equation}\aligned\label{DeNSiTh*}
\int_{B_{R}(p)}\left\lan\na\r_{\Si},\na\phi\right\ran\ge-\int_{B_{R}(p)}\phi\Th
\endaligned
\end{equation}
for any Lipschitz function $\phi\ge0$ with compact support in $B_{R}(p)\setminus \Si$, where we omitted the standard volume element of $B_R(p)$ in the above integration.

For an open $U$ in $N$,
let $\mathrm{Lip}_0(U)$ denote the space including all Lipschitz functions $f$ on $\overline{U}$ with $f=0$ on $\p U$.
Let $L$ denote a linear operator on the Sobolev space $W^{1,1}(B_{R}(p))$ defined by
\begin{equation}\aligned\label{Lphi}
L\phi=\int_{B_{R}(p)}\lan\na\phi,\na\r_\Si\ran+\int_{B_{R}(p)}\phi\Th
\endaligned
\end{equation}
for each function $\phi\in W^{1,1}(B_{R}(p))$.
From \eqref{DeNSiTh*}, we have
\begin{equation}\aligned
L\phi\ge0\qquad \mathrm{for\ each}\ \ \phi\ge0,\ \phi\in \mathrm{Lip}_0(B_{R}(p)\setminus\Si),
\endaligned
\end{equation}
which implies
\begin{equation}\aligned\label{phi1phi2}
L\phi_1\le L\phi_2\qquad \mathrm{for\ each}\ 0\le\phi_1\le\phi_2\ \mathrm{with}\ \phi_1,\phi_2\in \mathrm{Lip}_0(B_{R}(p)\setminus\Si).
\endaligned
\end{equation}

For a set $K$ with $\overline{K}\subset N$, let $B_t(K)$ be the $t$-tubular neighborhood of $K$ in $N$, i.e., $B_t(K)=\{x\in N|\,\inf_{y\in K}d(x,y)<t\}$.
For a bounded open set $\Om$ in $N$, let $\mathrm{Per}(\Om)$ denote the perimeter of $\Om$ defined by (see \cite{Gi} for the Euclidean case)
\begin{equation}\aligned
\mathrm{Per}(\Om)=\sup\left\{\int_\Om \mathrm{div}_N \mathscr{X}|\, \mathscr{X}\in \G^1_c(TN),\, |\mathscr{X}|\le1\ \mathrm{on}\ N\right\},
\endaligned
\end{equation}
where $\mathrm{div}_N$ denotes the divergence of $N$ w.r.t. its metric, $\G^1_c(TN)$ denotes the space containing all $C^1$ tangent vector fields on $N$ with compact supports.
Suppose $\mathrm{Per}(\Om)<\infty$.
Let $\sigma\in C^\infty_c((-1,1))$ be a symmetric function with $\int_{\R^{n+1}}\si(|z|^2)dz=1$, $\si(0)>0$, and $\si_\ep(x,y)=\ep^{-n-1}\si(\ep^{-2}d^2(x,y))$ for all $x,y\in N$, where $d(\cdot,\cdot)$ is the distance function on $N$. Then
\begin{equation}\aligned
\lim_{\ep\to0}\int_{y\in N}\si_\ep(x,y)=\int_{\R^{n+1}}\si(|z|^2)dz=1.
\endaligned
\end{equation}
Let $\chi_{_\Om}$ be the characteristic function of $\Om$.
For any small $\ep\in(0,1]$, let $w_{\ep}=w_{\ep,\Om}$ be a convolution of $\chi_{_\Om}$ and $\si_\ep$ defined by
\begin{equation}\aligned\label{wep}
w_\ep(x)=(\chi_{_\Om}*\si_\ep)(x)=\int_{y\in N} \chi_{_\Om}(y)\si_\ep(x,y)=\int_{y\in\Om}\si_\ep(x,y).
\endaligned
\end{equation}
Since the function $d^2(x,y)$ is smooth for small $d(x,y)$,
then $w_\ep$ is a smooth function with compact support in $\overline{B_\ep(\Om)}$. From \cite{AFP} or \cite{Gi}, we have
\begin{equation}\aligned
\mathrm{Per}(\Om)=\lim_{\ep\to0}\int_N\left|\na w_\ep\right|^2.
\endaligned
\end{equation}
By Sard's theorem, for every small $\ep>0$, $\p\{x\in N|\ w_\ep(x)>t\}$ is smooth for almost every $t$.
From the co-area formula and the semi-continuity of functions of bounded variation, there is a sequence $t_i\to0$ such that $\Om_i\triangleq\{x\in N|\ w_{t_i}(x)>1/i\}$ has smooth boundary with $\overline{\Om_i}\subset\overline{\Om}\setminus\p\overline{\Om}$, $\mathcal{H}^n(\p\Om_i\cap\Si)=0$ and (see \cite{AFP} or \cite{Gi})
\begin{equation}\aligned\label{perOmOmi}
\mathrm{Per}(\Om)=\lim_{i\to\infty}\mathcal{H}^n(\p\Om_i).
\endaligned
\end{equation}

For any real function $f$, we denote $f^+=\max\{0,f\}$ and $f^-=\max\{0,-f\}$.
\begin{lemma}\label{Radon}
For an open set $\Om\subset B_R(p)$ and a function $\varphi\in \mathrm{Lip}_0(\Om)$, we have
\begin{equation}\aligned\label{OmDvarDr}
L\varphi\le\sup_\Om\varphi^+\left(\mathrm{Per}(\Om)+2\mathcal{H}^n\left(\Si\cap\overline{\Om}\setminus\p\overline{\Om}\right)+\int_{\Om}|\Th|\right)
+2\left(\sup_{\Si\cap\Om}\varphi^-\right)\mathcal{H}^n\left(\Si\cap\Om\right).
\endaligned
\end{equation}
\end{lemma}
\begin{proof}
We assume $\mathrm{Per}(\Om)<\infty$, or else \eqref{OmDvarDr} holds automatically from \eqref{phi1phi2}.
Let $\Om_i$ be as above. For any $\phi\in \mathrm{Lip}_0(\Om\setminus\Si)$, there is a sequence of functions $\phi_i\in \mathrm{Lip}_0(\Om_i\setminus\Si)$ with $\inf_{\Om}\phi\le\inf_{\Om_i}\phi_i\le\sup_{\Om_i}\phi_i\le \sup_{\Om}\phi$
such that
\begin{equation}\aligned\label{Lvphii}
L\phi=\lim_{i\to\infty}L\phi_i.
\endaligned
\end{equation}
In fact, $\phi_i$ can be chosen as $\e_i\phi$ with the Lipschitz function $\e_i$ satisfying $\e_i=1$ on $\Om_i\setminus B_{\de_i}(\p\Om_i)$, $\e_i=\f1{\de_i}\r_{\Om_i}$ on $\Om_i\cap B_{\de_i}(\p\Om_i)$, and $\e_i=0$ outside $\Om_i$ for some sequence $\de_i\to0$ as $i\to\infty$ so that $\de_i^{-1}\mathcal{H}^{n+1}(\Om_i\cap B_{\de_i}(\p\Om_i))<(1+i^{-1})\mathcal{H}^n(\p\Om_i)$ for each $i$.

Let $\Om_{i,\ep}=\Om_i\setminus \overline{B_\ep(\p\Om_i\cup\Si)}$ for any $\ep>0$.
Let $\e_{i,\ep}=1-\f1\ep\r_{\Om_{i,\ep}}$ on $\Om_s$ and $\e_{i,\ep}=0$ outside $\Om_i$, then $\e_{i,\ep}$ is Lipschitz on $N$ with the Lipschitz constant $1/\ep$.
We can assume $\sup_\Om\phi^+>0$, or else \eqref{OmDvarDr} holds clearly from \eqref{phi1phi2}. Hence, without loss of generality we may assume $\sup_\Om\phi_i^+>0$ for each integer $i\ge1$.
Since $\phi_i\in \mathrm{Lip}_0(\Om_i)$, there is a sequence $\ep_{i,j}>0$ with $\ep_{i,j}\to0$ as $j\to\infty$ so that $\phi_i\le\e_{i,\ep_{i,j}}\sup_{\Om}\phi^+_i\le\e_{i,\ep_{i,j}}\sup_{\Om}\phi^+$.
From \eqref{Lphi}\eqref{phi1phi2}, we have
\begin{equation}\aligned\label{Omphi+***}
\f1{\sup_{\Om}\phi^+}L\phi_i\le& L\e_{s,\ep_{i,j}}=-\f1{\ep_{i,j}}\int_{\Om_i}\lan\na\r_{\Om_{i,\ep_{i,j}}},\na\r_\Si\ran+\int_{\Om_i}\e_{i,\ep_{i,j}}\Th\\
\le&\f1{\ep_{i,j}}\mathcal{H}^{n+1}(\Om_i\setminus \Om_{i,\ep_{i,j}})+\int_{\Om_i}|\Th|.
\endaligned
\end{equation}
With \eqref{locVolSi}, $\mathcal{H}^n(\p\Om_i\cap\Si)=0$ and Theorem 2.104 in \cite{AFP} by Ambrosio-Fusco-Pallara,
letting $j\rightarrow\infty$ in \eqref{Omphi+***} gives
\begin{equation}\aligned\label{Lphisupp}
L\phi_i\le\sup_{\Om}\phi^+\left(\mathcal{H}^n(\p \Om_i)+2\mathcal{H}^n(\Om_i\cap\Si)+\int_{\Om_i}|\Th|\right).
\endaligned
\end{equation}
Combining $\overline{\Om_i}\subset\overline{\Om}\setminus\p\overline{\Om}$ and \eqref{perOmOmi}\eqref{Lvphii}, letting $i\to\infty$ in \eqref{Lphisupp} infers
\begin{equation}\aligned\label{UPLphi}
L\phi\le\sup_{\Om}\phi^+\left(\mathrm{Per}(\Om)+2\mathcal{H}^n\left(\Si\cap\overline{\Om}\setminus\p\overline{\Om}\right)+\int_{\Om}|\Th|\right).
\endaligned
\end{equation}

For any $\ep\in(0,1)$, let $\tilde{\e}_\ep$ be a nonnegative Lipschitz function on $B_{R}(p)$ with $\tilde{\e}_\ep\equiv1$ on
$B_{R}(p)\setminus B_{\ep}(\Si)$, $\tilde{\e}_\ep=\r_\Si/\ep$ on $B_{R}(p)\cap B_{\ep}(\Si)$.
For any $\varphi\in \mathrm{Lip}_0(\Om)$, from \eqref{locVolSi} and the proof of Theorem 2.104 in \cite{AFP}, it follows that
\begin{equation}\aligned\label{AFPSipBRp}
\limsup_{\ep\rightarrow0}\int_{\Om}-\varphi\lan\na\tilde{\e}_\ep,\na\r_\Si\ran\le2\left(\sup_{\Si\cap\Om}\varphi^-\right)\mathcal{H}^n\left(\Si\cap\Om\right).
\endaligned
\end{equation}
With \eqref{UPLphi}, we have
\begin{equation}\aligned\label{varphiuprSi}
L(\tilde{\e}_\ep\varphi)\le&\sup_\Om\varphi^+\left(\mathrm{Per}(\Om)+2\mathcal{H}^n\left(\Si\cap\overline{\Om}\setminus\p\overline{\Om}\right)+\int_{\Om}|\Th|\right).
\endaligned
\end{equation}
Since
\begin{equation}\aligned\label{navarphiee}
\int_{\Om}\tilde{\e}_\ep\lan\na\varphi,\na\r_\Si\ran+\int_{\Om}\tilde{\e}_\ep\varphi\Th=&\int_{\Om}\lan\na(\tilde{\e}_\ep\varphi),\na\r_\Si\ran-\int_{\Om}\varphi\lan\na\tilde{\e}_\ep,\na\r_\Si\ran+\int_{\Om}\tilde{\e}_\ep\varphi\Th\\
=&L(\tilde{\e}_\ep\varphi)-\int_{\Om}\varphi\lan\na\tilde{\e}_\ep,\na\r_\Si\ran,
\endaligned
\end{equation}
letting $\ep\to0$ in \eqref{navarphiee},  with \eqref{AFPSipBRp}\eqref{varphiuprSi} we get
\begin{equation}\aligned\label{Lvarphi+-}
L\varphi\le\sup_\Om\varphi^+\left(\mathrm{Per}(\Om)+2\mathcal{H}^n\left(\Si\cap\overline{\Om}\setminus\p\overline{\Om}\right)+\int_{\Om}|\Th|\right)
+2\left(\sup_{\Si\cap\Om}\varphi^-\right)\mathcal{H}^n\left(\Si\cap\Om\right).
\endaligned
\end{equation}
This completes the proof.
\end{proof}
For an open $U$ in $N$,
let $\mathrm{Lip}_c(U)$ denote the space including all Lipschitz functions on $U$ with compact supports in $U$.
\begin{lemma}\label{RiszeR}
There is a Radon measure $\mu$ on $B_R(p)$ such that
\begin{equation}\aligned
\int_{B_R(p)} fd\mu=\sup_{\phi\in\mathrm{Lip}_c(B_R(p)),\, 0\le\phi\le f}L\phi
\endaligned
\end{equation}
for any nonnegative continuous function $f$ on $B_R(p)$ with compact support in $B_R(p)$.
Moreover, for any open $\Om\subset B_R(p)$
\begin{equation}\aligned\label{muOm}
\mu(\Om)\le\mathrm{Per}(\Om)+2\mathcal{H}^n\left(\Si\cap\overline{\Om}\setminus\p\overline{\Om}\right)+\int_{\Om}|\Th|.
\endaligned
\end{equation}
\end{lemma}
\begin{proof}
For any open $\Om\subset B_R(p)$, let $\mathcal{K}_+(\Om)$ denote the set containing all the nonnegative continuous functions with compact supports in $\Om$.
Let $\widetilde{L}:\, \mathcal{K}_+(\Om)\rightarrow[0,\infty)$ defined by
\begin{equation}\aligned
\widetilde{L}(f)=\sup_{\phi\in \mathrm{Lip}_c(\Om),\, 0\le\phi\le f}L\phi.
\endaligned
\end{equation}
For each $f_1,f_2\in \mathcal{K}_+(\Om)$ and each $\ep>0$, there are functions $\phi_i\in \mathrm{Lip}_c(\Om),\, 0\le\phi_i\le f_i$ for $i=1,2$ such that
\begin{equation}\aligned
\widetilde{L}f_1\le L\phi_1+\ep,\qquad \widetilde{L}f_2\le L\phi_2+\ep.
\endaligned
\end{equation}
Hence
\begin{equation}\aligned\label{wtf1f2ep}
\widetilde{L}f_1+\widetilde{L}f_2\le L\phi_1+L\phi_2+2\ep\le \widetilde{L}(f_1+f_2)+2\ep.
\endaligned
\end{equation}

On the other hand, there is a function $\phi_0\in \mathrm{Lip}_c(\Om)$ with $0\le\phi_0\le f_1+f_2$ such that 
\begin{equation}\aligned\label{wtLf1f2ep}
\widetilde{L}(f_1+f_2)\le L\phi_0+\ep.
\endaligned
\end{equation}
It's easy to check that $\phi_0f_1/(f_1+f_2),\phi_0f_2/(f_1+f_2)$ are continuous functions with supports both in $\mathrm{spt}\phi_0$.
By the construction of mollifiers, for any $\de>0$ with $B_{2\de}(\mathrm{spt}\phi_0)\subset \Om$ there are smooth functions $f_{1,\de},f_{2,\de}\ge0$ with supports in $B_\de(\mathrm{spt}\phi_0)$ such that
\begin{equation}\aligned\label{fidephi0f1f2}
\left|f_{i,\de}-\f{\phi_0f_i}{f_1+f_2}\right|\le \f{\de}2\qquad \mathrm{for}\ i=1,2.
\endaligned
\end{equation}
From $0\le\phi_i\le f_i$ with $i=1,2$ and $0\le\phi_0\le f_1+f_2$, we have
$f_{i,\de}\le f_i+\f{\de}2$.
Let $\phi_\de\in C^\infty(\Om,[0,\de/2])$ satisfy $\phi_\de=\de/2$ on $B_\de(\mathrm{spt}\phi_0)$ and $\phi_\de=0$ on $\Om\setminus B_{3\de/2}(\mathrm{spt}\phi_0)$.
For each $i=1,2$, let $f^*_{i,\de}=\max\{0,f_{i,\de}-\phi_\de\}$, then $f^*_{i,\de}\in \mathrm{Lip}_c(\Om)$.
Since $f_{i,\de}$ has the support in $B_\de(\mathrm{spt}\phi_0)$ for $i=1,2$, it follows that
\begin{equation}\aligned\label{0fide*fi}
0\le f_{i,\de}^*\le f_i\qquad \mathrm{on}\ \Om.
\endaligned
\end{equation}
From $\phi_0\ge0$, $0\le\phi_\de\le\de/2$ and \eqref{fidephi0f1f2}, we have 
\begin{equation}\aligned\label{phi0*}
\left|\sum_{i=1}^2f^*_{i,\de}-\phi_0\right|\le\left|\sum_{i=1}^2(f_{i,\de}-\phi_\de)-\phi_0\right|\le\sum_{i=1}^2\left|f_{i,\de}-\phi_\de-\f{\phi_0f_i}{f_1+f_2}\right|\le \de+\de=2\de.
\endaligned
\end{equation}
Denote $\phi_0^*=\phi_0-\sum_{i=1}^2f^*_{i,\de}$.
With Lemma \ref{Radon} and \eqref{wtLf1f2ep}\eqref{0fide*fi}\eqref{phi0*}, we get
\begin{equation}\aligned\label{wwtf1f2de}
&\widetilde{L}(f_1+f_2)-\ep\le L\phi_0= Lf_{1,\de}^*+Lf_{2,\de}^*+L\phi_0^*\\
\le& \widetilde{L}f_1+\widetilde{L}f_2+2\de\left(\mathrm{Per}(\Om)+2\mathcal{H}^n\left(\Si\cap\Om\setminus\p\Om\right)+\int_{\Om}|\Th|+2\mathcal{H}^n\left(\Si\cap\Om\right)\right).
\endaligned
\end{equation}
Let $\de\rightarrow0$ and $\ep\rightarrow0$ in \eqref{wtf1f2ep}\eqref{wwtf1f2de}, we deduce
\begin{equation}\aligned
\widetilde{L}f_1+\widetilde{L}f_2=\widetilde{L}(f_1+f_2).
\endaligned
\end{equation}
Clearly, $\widetilde{L}(cf)=c\widetilde{L}f\ge0$ for any constant $c\ge0$ and $f\in\mathcal{K}_+(\Om)$. 
From Lemma \ref{Radon} and Riesz representation theorem, there is a Radon measure $\mu$ on $B_R(p)$ such that
\begin{equation}\aligned\label{Lfdmu}
\widetilde{L}f=\int_{B_R(p)} fd\mu\qquad \mathrm{for\ any}\ f\in\mathcal{K}_+(B_R(p)).
\endaligned
\end{equation}

For any open $\Om'\subset\subset\Om$, let $\xi_s=1-\f1s\r_{\Om'}$ on $B_s(\Om')$, and $\xi_s=0$ outside $B_s(\Om')$ for small $s>0$.
We assume $B_s(\Om')\subset\subset\Om$.
Then $\mu(\Om')\le \widetilde{L}\xi_s$.
From \eqref{OmDvarDr}, we get
\begin{equation}\aligned
\mu(\Om')\le\mathrm{Per}(\Om)+2\mathcal{H}^n\left(\Si\cap\overline{\Om}\setminus\p\overline{\Om}\right)+\int_{\Om}|\Th|.
\endaligned
\end{equation}
Forcing $\Om'\to\Om$ in the above inequality implies \eqref{muOm}.
\end{proof}

We define a Radon measure $\mu_*$ by
\begin{equation}\aligned
\mu_*=\mu+2\mathcal{H}^n\llcorner\Si,
\endaligned
\end{equation}
i.e., $\mu_*(K)=\mu(K)+2\mathcal{H}^n(K\cap\Si)$ for any Borel set $K\subset B_R(p)$.
\begin{corollary}\label{|Lvarphi|mu*}
For an open $\Om\subset B_R(p)$, and a function $\varphi\in \mathrm{Lip}_0(\Om)$, we have
\begin{equation}\aligned
|L\varphi|\le\sup_{\Om}|\varphi|\,\mu_*(\Om).
\endaligned
\end{equation}
\end{corollary}
\begin{proof}
Denote $\varphi=\varphi^+-\varphi^-$.\
Let $\Om_i$ and $\e_i$ be the ones as in the proof of Lemma \ref{Radon}. Put $\varphi_i=\e_i\varphi^+\in\mathrm{Lip}_c(\Om)$. With \eqref{Lvphii} and Lemma \ref{RiszeR}, we get
\begin{equation}\aligned\label{Lv+0}
L\varphi^+=\lim_{i\to\infty}L\varphi_i\le\sup_{\Om}\varphi^+\,\mu(\Om).
\endaligned
\end{equation}
From Lemma \ref{Radon}, it follows that
\begin{equation}\aligned\label{Lv-0}
L(-\varphi^-)\le2\sup_{\Om}\varphi^-\,\mathcal{H}^n(\Om\cap\Si).
\endaligned
\end{equation}
Combining \eqref{Lv+0}\eqref{Lv-0}, we get
\begin{equation}\aligned\label{Lvmu*}
L\varphi=L(\varphi^+-\varphi^-)\le\sup_{\Om}|\varphi|\left(\mu(\Om)+2\mathcal{H}^n(\Om\cap\Si)\right)=\sup_{\Om}|\varphi|\,\mu_*(\Om).
\endaligned
\end{equation}
By considering $-\varphi$ instead of $\varphi$, we complete the proof.
\end{proof}
\begin{remark}
In Lemma \ref{Radon}, we have obtained the estimate of $L\varphi$, but the right hand of \eqref{OmDvarDr} is not necessarily non-decreasing as $\Om$ increases because of the term $\mathrm{Per}(\Om)$. However, the upper bound $\mu_*(\Om)$ does not decrease as $\Om$ increases.
\end{remark}

\begin{proposition}
For any ball $B_r(z)\subset B_R(p)$,
\begin{equation}\aligned\label{muBrzSi}
\mu_*(B_r(z))\le \f{(n+1)\k}{\tanh(\k r)}\mathcal{H}^{n+1}(B_r(z))+4\mathcal{H}^{n}(B_r(z)\cap \Si)+\int_{B_r(z)}|\Th|.
\endaligned
\end{equation}
\end{proposition}
\begin{proof}
For any $r>0$, with integrating by parts we have
\begin{equation}\aligned
&n\k\int_0^r\sinh^n(\k t)dt=\int_0^r\tanh(\k t)d\sinh^n(\k t)\\
=&\sinh^n(\k r)\tanh(\k r)-\k\int_0^r\sinh^n(\k t)\left(1-\tanh^2(\k t)\right)dt,
\endaligned
\end{equation}
which implies
\begin{equation}\aligned
(n+1)\k\int_0^r\sinh^n(\k t)dt\ge\sinh^n(\k r)\tanh(\k r).
\endaligned
\end{equation}
With Bishop-Gromov volume comparison, we have
\begin{equation}\aligned\label{HnpBrz}
\mathcal{H}^{n}(\p B_r(z))\le \f{\sinh^n(\k r)}{\int_0^r\sinh^n(\k t)dt}\mathcal{H}^{n+1}(B_r(z))\le \f{(n+1)\k}{\tanh(\k r)}\mathcal{H}^{n+1}(B_r(z)).
\endaligned
\end{equation}
Combining with \eqref{muOm}\eqref{HnpBrz} and the definition of $\mu_*$, we immediately get \eqref{muBrzSi}.
\end{proof}

\begin{remark}
Let $S$ be a countably $n$-rectifiable set in $N$ with $\mathcal{H}^n(S)<\infty$, and there are a Radon measure $\mathfrak{m}$ on $N$ and constants $\g,\g'>0$ such that $\mathfrak{m}$ is absolutely continuous with respect to $\mathcal{H}^n$ and
\begin{equation}\aligned\label{mmgrn}
\mathfrak{m}(B_r(x))\ge\g r^n \ for\ any\ x\in S\cap B_R(p),\ and\ any\ r\in(0,\g').
\endaligned
\end{equation}
We suppose that there is a $L^1$-function $f$ on $B_R(p)$ such that $\De_N\r_S\le f$ on $B_R(p)\setminus S$
in the distribution sense. Then our argument in $\S 3$ and $\S 4$ still works provided $\Si$ is replaced by $S$ here.
Without the condition \eqref{mmgrn}, our argument in $\S 3$ and $\S 4$ also works under a more careful covering for $S$, and for any open $\Om\subset B_R(p)$ the coefficient of $\mathcal{H}^n(S\cap\Om)$ will be replaced by a constant depending only on $n,\k$. 
\end{remark}

\section{An angle estimate for distance functions from hypersurfaces}

Let $N$ be the $(n+1)$-dimensional manifold, and $\Si$ be the subset in $N$ defined at the beginning of $\S 3$.
Now let us approach $\r_\Si$ by harmonic functions analog to \cite{C}.
We define a Radon measure $\mu^*$ on $B_R(p)$ by letting
\begin{equation}\aligned\label{DEFmu*}
\mu^*(K)=\mu_*(K)+\int_K|\Th|
\endaligned
\end{equation}
for any Borel set $K\subset B_R(p)$.
\begin{lemma}\label{nab-bz|}
For any $B_{2r}(z)\subset B_R(p)$, there is a harmonic function $\mathbf{b}$ on $B_{2r}(z)$ with $|\mathbf{b}|\le 2r$ on $B_{2r}(z)$ such that
\begin{equation}\aligned
\int_{B_{r}(z)}\left|\na\mathbf{b}-\na \r_\Si\right|^2\le4r\mu^*(B_{2r}(z)).
\endaligned
\end{equation}
\end{lemma}
\begin{proof}
We fix a point $z\in B_R(p)$, and define a Lipschitz function
$$b_z=\r_\Si-\r_\Si(z)\qquad \mathrm{on}\ \ B_{R+R'}(p).$$
Then  for each $x\in B_{R+R'}(p)$
\begin{equation}\aligned\label{bzxdxz}
|b_z(x)|=|\r_\Si(x)-\r_\Si(z)|\le d(x,z).
\endaligned
\end{equation}
For any $\ep\in(0,1]$, let $b_{z,\ep}$ be a mollifier of $b_z$ defined by
$$b_{z,\ep}(x)=(b_z*\si_\ep)(x)=\int_{y\in N} b_z(y)\si_\ep(x,y),$$
where $\si_\ep$ is the function defined in \eqref{wep}.
Since $\si_\ep$ is smooth for small $\ep>0$, it follows that $b_{z,\ep}\in C^\infty$ for small $\ep>0$.
Then for any $U\subset B_R(p)$ we have
\begin{equation}\aligned\label{nabznabzep}
\lim_{\ep\to0}\int_{U}\left|\na b_z-\na b_{z,\ep}\right|^2=0.
\endaligned
\end{equation}
We fix a ball $B_{2r}(z)\subset B_R(p)$. For every small $\ep>0$, 
let $U_\ep$ be an open set in $B_{2r}(z)$ with smooth $\p U_\ep$ such that $U_\ep$ converges to $B_{2r}(z)$ as $\ep\to0$.
Let $\mathbf{b}_\ep$ denote the harmonic function on $U_\ep$ with $\mathbf{b}_\ep=b_{z,\ep}$ on $\p U_\ep$.
From Schauder theory of elliptic equations, $\mathbf{b}_\ep\in C^\infty(\overline{U_\ep})$.
From the maximum principle, we have
$$\sup_{U_\ep}|\mathbf{b}_\ep|=\sup_{\p U_\ep}|\mathbf{b}_\ep|=\sup_{\p U_\ep}|b_{z,\ep}|.$$
Since $b_{z,\ep}\to b_z$ uniformly, $|b_z|\le d(\cdot,z)$ on $B_{2r}(p)$, and $U_\ep\subset B_{2r}(z)$, then there is a function $\psi_\ep>0$ with $\lim_{\ep\to0}\psi_\ep=0$ such that
\begin{equation}\aligned\label{supbepbzep}
\sup_{U_\ep}\left|\mathbf{b}_\ep-b_{z,\ep}\right|\le\sup_{U_\ep}|\mathbf{b}_\ep|+\sup_{U_\ep}|b_{z,\ep}|\le 4r+\psi_\ep.
\endaligned
\end{equation}
Let $U_{\ep,\tau}=\{x\in U_\ep|\r_{\p U_\ep}(x)>\tau\}$
for any small $\tau>0$. Let $\e_\tau=1$ on $U_{\ep,\tau}$, $\e_\tau=1-\f1\tau\r_{\p U_\ep}$ on $U_\ep\setminus U_{\ep,\tau}$, and $\e_\tau=0$ outside $U_\ep$.
For any $\de\in(0,1]$, by the definition of $L$ in \eqref{Lphi} one has
\begin{equation}\aligned
&\int_{U_\ep}\left|\na\mathbf{b}_\ep-\na b_{z,\ep}\right|^2=-\int_{U_\ep}\lan\na(\mathbf{b}_\ep-b_{z,\ep}),\na b_{z,\ep}\ran+\int_{U_\ep}\mathrm{div}_N\left((\mathbf{b}_\ep-b_{z,\ep})\na\mathbf{b}_\ep\right)\\
=&-\int_{U_\ep}\lan\na(\mathbf{b}_\ep-b_{z,\ep}),\na b_{z,\ep}\ran=-\int_{U_\ep}\lan\na(\mathbf{b}_\ep-b_{z,\ep}),\na b_z\ran+\int_{U_\ep}\lan\na(\mathbf{b}_\ep-b_{z,\ep}),\na(b_z-b_{z,\ep})\ran\\
\le&-\int_{U_\ep}\lan\na(\mathbf{b}_\ep-b_{z,\ep}),\na b_z\ran+\de\int_{U_\ep}|\na(\mathbf{b}_\ep-b_{z,\ep})|^2+\f1{\de}\int_{U_\ep}|\na(b_z-b_{z,\ep})|^2,
\endaligned
\end{equation}
which implies
\begin{equation}\aligned\label{Uepbepbzep}
&(1-\de)\int_{U_\ep}\left|\na\mathbf{b}_\ep-\na b_{z,\ep}\right|^2\le-\int_{U_\ep}\lan\na(\mathbf{b}_\ep-b_{z,\ep}),\na b_z\ran+\f1{\de}\int_{U_\ep}|\na(b_z-b_{z,\ep})|^2.
\endaligned
\end{equation}
By the definition of the operator $L$ in \eqref{Lphi}, 
\begin{equation}\aligned\label{Uepbepbzep'}
&-\int_{U_\ep}\e_\tau\lan\na(\mathbf{b}_\ep-b_{z,\ep}),\na b_z\ran\\
=&-\int_{U_\ep}\lan\na(\e_\tau(\mathbf{b}_\ep-b_{z,\ep})),\na b_z\ran+\int_{U_\ep}(\mathbf{b}_\ep-b_{z,\ep})\lan\na\e_\tau,\na b_z\ran\\
\le& \int_{U_\ep}\e_\tau(\mathbf{b}_\ep-b_{z,\ep})\Th-L(\e_\tau(\mathbf{b}_\ep-b_{z,\ep}))+\f1{\tau}\int_{U_\ep\setminus U_{\ep,\tau}}\left|\mathbf{b}_\ep-b_{z,\ep}\right|.
\endaligned
\end{equation}
From Corollary \ref{|Lvarphi|mu*} and \eqref{supbepbzep},
\begin{equation}\aligned\label{Letbepbzep}
-L(\e_\tau(\mathbf{b}_\ep-b_{z,\ep}))\le\sup_{U_\ep}\left|\mathbf{b}_\ep-b_{z,\ep}\right|\mu_*(U_\ep)\le(4r+\psi_\ep)\mu_*(U_\ep).
\endaligned
\end{equation}
Since $\mathbf{b}_\ep\in C^\infty(\overline{U_\ep})$ and $\mathbf{b}_\ep=b_{z,\ep}$ on $\p U_\ep$, then 
\begin{equation}\aligned
\lim_{\tau\to0}\sup_{U_\ep\setminus U_{\ep,\tau}}\left|\mathbf{b}_\ep-b_{z,\ep}\right|=0,
\endaligned
\end{equation}
and
\begin{equation}\aligned\label{UepUeptbepbzep}
\lim_{\tau\to0}\f1{\tau}\int_{U_\ep\setminus U_{\ep,\tau}}\left|\mathbf{b}_\ep-b_{z,\ep}\right|
\le\lim_{\tau\to0}\sup_{U_\ep\setminus U_{\ep,\tau}}\left|\mathbf{b}_\ep-b_{z,\ep}\right|\f{\mathcal{H}^{n+1}\left(U_\ep\setminus U_{\ep,\tau}\right)}{\tau}=0.
\endaligned
\end{equation}
Letting $\tau\to0$ in \eqref{Uepbepbzep'}\eqref{Letbepbzep}, combining with \eqref{supbepbzep}\eqref{UepUeptbepbzep} we have
\begin{equation}\aligned
-\int_{U_\ep}\lan\na(\mathbf{b}_\ep-b_{z,\ep}),\na b_z\ran\le& (4r+\psi_\ep)\int_{U_\ep}|\Th|+(4r+\psi_\ep)\mu_*(U_\ep).
\endaligned
\end{equation}
Combining \eqref{Uepbepbzep}, we get
\begin{equation}\aligned\label{Brznambepbzep}
(1-\de)\int_{U_\ep}\left|\na\mathbf{b}_\ep-\na b_{z,\ep}\right|^2\le& (4r+\psi_\ep)\int_{U_\ep}|\Th|+(4r+\psi_\ep)\mu_*(U_\ep)+\f1{\de}\int_{U_\ep}|\na(b_z-b_{z,\ep})|^2.
\endaligned
\end{equation}
By the maximum principle and the definition of $\mathbf{b}_{\ep}$,
there is a sequence $\ep_i\to0$ such that $\mathbf{b}_{\ep_i}$ converges to a harmonic function $\mathbf{b}$ on $B_{2r}(z)$ with $|\mathbf{b}|\le 2r$ on $B_{2r}(z)$. In particular, $\mathbf{b}_{\ep_i}$ converges to $\mathbf{b}$ uniformly on any compact subset of $B_{2r}(z)$. 
Hence $\left|\na\mathbf{b}_{\ep_i}-\na\mathbf{b}\right|\to0$ uniformly on $\overline{B_{r}(z)}$. From \eqref{nabznabzep} and \eqref{Brznambepbzep} with $\ep=\ep_i$,  letting $\ep_i\to0$ we get
\begin{equation}\aligned
(1-\de)\int_{B_{r}(z)}\left|\na\mathbf{b}-\na b_z\right|^2\le& 4r\int_{B_{2r}(z)}|\Th|+4r\mu_*(B_{2r}(z)).
\endaligned
\end{equation}
By the definition of $\mu^*$ in \eqref{DEFmu*}, we complete the proof by letting $\de\to0$ in the above inequality.
\end{proof}
From the Cheng-Yau gradient estimate \cite{CY} and \eqref{bzxdxz}, for the harmonic function $\mathbf{b}$ in Lemma \ref{nab-bz|}
\begin{equation}\aligned
\sup_{B_{3r/2}(z)}\left|\na \mathbf{b}\right|\le \f{c_{n,\k r}}{2r}\sup_{B_{2r}(z)}|\mathbf{b}|\le c_{n,\k r},
\endaligned
\end{equation}
where $c_{n,\k r}$ is a general constant depending on $n,\k r$.
Let 'Hess' denote the Hessian matrix for $C^2$-functions on $N$.
Combining the Bochner formula
\begin{equation}\aligned
\f12\De_N\left|\na \mathbf{b}\right|^2=\left|\mathrm{Hess}_{\mathbf{b}}\right|^2+\mathrm{Ric}_N(\na \mathbf{b},\na \mathbf{b})
\ge\left|\mathrm{Hess}_{\mathbf{b}}\right|^2-n\k^2\left|\na \mathbf{b}\right|^2,
\endaligned
\end{equation}
we get the following Hessian estimate (see Lemma 1.12 in \cite{C})
\begin{equation}\aligned\label{Hessb}
\int_{B_r(z)}\left|\mathrm{Hess}_{\mathbf{b}}\right|^2\le\f{c_{n,\k r}}{r^2}\mathcal{H}^{n+1}(B_r(z)).
\endaligned
\end{equation}

Let $\La_r>0$ be a constant (depending on $r$) such that
\begin{equation}\aligned\label{SiLa}
4\mathcal{H}^n(B_{2r}(z)\cap\Si)+2\int_{B_{2r}(z)}|\Th|\le \f{\La_r}r\mathcal{H}^{n+1}(B_{2r}(z)).
\endaligned
\end{equation}
From \eqref{muBrzSi}\eqref{SiLa}, we have
\begin{equation}\aligned\label{ThmuLaCONT}
\mu^*(B_{2r}(z))=\int_{B_{2r}(z)}|\Th|+\mu_*(B_{2r}(z))\le \left(\f{c_{n,\k r}}{r}+\f{\La_r}r\right)\mathcal{H}^{n+1}(B_{2r}(z)).
\endaligned
\end{equation}
Inspired by Lemma 1.23 in \cite{C}, we can deduce the following estimates.
\begin{lemma}\label{Decomb}
For any $\de>0$, there is a constant $\th_*=\th_{n,\k r,\La_r,\de}$ depending on $n,\k r,\La_r,\de$ such that
for each $0<\th\le\th_*$ and $B_{4r}(z)\subset B_R(p)$, there exist finitely many balls $B_{\th r}(p_i)\subset B_r(z)$ for $i=1,\cdots,N_*$, and harmonic functions $\mathbf{b}_i$ satisfying
\begin{equation}\aligned\label{Bthrpi}
\mathcal{H}^{n+1}\left(\cup_{1\le i\le N_*}B_{\th r}(p_i)\right)\ge(1-\de)\mathcal{H}^{n+1}\left(B_{r}(z)\right),
\endaligned
\end{equation}
\begin{equation}\aligned
\fint_{B_{2\th r}(p_i)}\left|\na\mathbf{b}_i-\na \r_\Si\right|^2\le \de,
\endaligned
\end{equation}
and
\begin{equation}\aligned
\fint_{B_{2\th r}(p_i)}\left|\mathrm{Hess}_{\mathbf{b}_i}\right|^2\le \f{\de}{\th^2r^2}.
\endaligned
\end{equation}
\end{lemma}
\begin{proof}
From Lemma 1.19 in \cite{C} and \eqref{ThmuLaCONT},
for any $\ep>0$ there are constants $\be_\ep\ge1$, $\th_\ep\le1$ and an integer $N_\ep\ge1$ depending on $n,\k r,\La_r,\ep$, and a constant $\la_*>0$ depending only on $n,\k r$ such that
for each $0<\th\le\th_\ep$, there exist finitely many balls $B_{\th r}(y_i)\subset B_r(z)$ for $i=1,\cdots,N_\ep$ with
\begin{equation}\aligned\label{layi}
\sum_{1\le i\le N_\ep}\mathcal{H}^{n+1}\left(B_{\th r}(y_i)\right)\le\la_*\mathcal{H}^{n+1}\left(B_{r}(z)\right),
\endaligned
\end{equation}
\begin{equation}\aligned\label{Bthryi}
\mathcal{H}^{n+1}\left(\cup_{1\le i\le N_\ep}B_{\th r}(y_i)\right)\ge(1-\ep)\mathcal{H}^{n+1}\left(B_{r}(z)\right),
\endaligned
\end{equation}
and
\begin{equation}\aligned\label{ndetanhmuyi}
\mu^*(B_{4\th r}(y_i))\le \f{\be_\ep}r\left(c_{n,\k r}+\La_r\right)\mathcal{H}^{n+1}(B_{2\th r}(y_i))
\endaligned
\end{equation}
for each $i=1,\cdots,N_\ep$.
From Lemma \ref{nab-bz|} and \eqref{ndetanhmuyi}, there is a harmonic function $\mathbf{b}_i$ on $B_{4\th r}(y_i)$ with $|\mathbf{b}_i|\le4\th r$ on $B_{4\th r}(y_i)$ such that
\begin{equation}\aligned\label{nabirSi}
\int_{B_{2\th r}(y_i)}\left|\na\mathbf{b}_i-\na \r_{\Si}\right|^2\le 8\th r\mu^*(B_{4\th r}(y_i))
\le 8\th\be_\ep \left(c_{n,\k r}+\La_r\right)\mathcal{H}^{n+1}(B_{2\th r}(y_i)).
\endaligned
\end{equation}
From \eqref{Hessb}, we have
\begin{equation}\aligned\label{Hessbi}
\fint_{B_{2\th r}(y_i)}\left|\mathrm{Hess}_{\mathbf{b}_i}\right|^2\le\f{c_{n,\k r}}{\th^2r^2}.
\endaligned
\end{equation}

For any fixed $i\in\{1,\cdots,N_\ep\}$ and $\th\in(0,\th_\ep]$, we use Lemma 1.19 in \cite{C} again for \eqref{nabirSi}\eqref{Hessbi} as follows. For each $0<\th'\le\th_\ep$
there exist finitely many balls $B_{\th'\th r}(x_{i,j})\subset B_{\th r}(y_i)$ for $j=1,\cdots,N_\ep$ with
\begin{equation}\aligned\label{Bthrxjyi}
\mathcal{H}^{n+1}\left(\cup_{1\le j\le N_\ep}B_{\th' \th r}(x_{i,j})\right)\ge(1-\ep)\mathcal{H}^{n+1}\left(B_{\th r}(y_i)\right),
\endaligned
\end{equation}
\begin{equation}\aligned\label{nabirSi*}
\fint_{B_{2\th'\th r}(x_{i,j})}\left|\na\mathbf{b}_i-\na \r_{\Si}\right|^2\le
 8\th\be_\ep^2 \left(c_{n,\k r}+\La_r\right),
\endaligned
\end{equation}
and
\begin{equation}\aligned\label{Hessbi*}
\fint_{B_{2\th'\th r}(x_{i,j})}\left|\mathrm{Hess}_{\mathbf{b}_i}\right|^2\le\f{c_{n,\k r}\be_\ep}{\th^2r^2}.
\endaligned
\end{equation}
Noting 
\begin{equation}\aligned
&\cup_{1\le i\le N_\ep}B_{\th r}(y_i)\setminus\cup_{1\le i,j\le N_\ep}B_{\th' \th r}(x_{i,j})\\
=&\cup_{1\le i\le N_\ep}\left(B_{\th r}(y_i)\cap\left(\cap_{1\le k,j\le N_\ep}(B_r(z)\setminus B_{\th' \th r}(x_{k,j}))\right)\right)\\
\subset&\cup_{1\le i\le N_\ep}\left(B_{\th r}(y_i)\cap\left(\cap_{1\le j\le N_\ep}(B_r(z)\setminus B_{\th' \th r}(x_{i,j}))\right)\right)\\
=&\cup_{1\le i\le N_\ep}\left(B_{\th r}(y_i)\setminus \cup_{1\le j\le N_\ep}B_{\th' \th r}(x_{i,j})\right).
\endaligned
\end{equation}
Combining \eqref{layi}\eqref{Bthryi}\eqref{Bthrxjyi}, we have
\begin{equation}\aligned\label{Bthrxij}
&\mathcal{H}^{n+1}\left(B_r(z)\right)-\mathcal{H}^{n+1}\left(\cup_{1\le i,j\le N_\ep}B_{\th' \th r}(x_{i,j})\right)\\\le&\ep\mathcal{H}^{n+1}\left(B_r(z)\right)+\mathcal{H}^{n+1}\left(\cup_{1\le i\le N_\ep}B_{\th r}(y_i)\right)-\mathcal{H}^{n+1}\left(\cup_{1\le i,j\le N_\ep}B_{\th' \th r}(x_{i,j})\right)\\
\le&\ep\mathcal{H}^{n+1}\left(B_r(z)\right)+\sum_{1\le i\le N_\ep}\mathcal{H}^{n+1}\left(B_{\th r}(y_i)\setminus \cup_{1\le j\le N_\ep}B_{\th' \th r}(x_{i,j})\right)\\
 \le&\ep\mathcal{H}^{n+1}\left(B_r(z)\right)+\ep\sum_{1\le i\le N_\ep}\mathcal{H}^{n+1}\left(B_{\th r}(y_i)\right)\le(1+\la_*)\ep\mathcal{H}^{n+1}\left(B_r(z)\right).
\endaligned
\end{equation}
Let $\de=(1+\la_*)\ep$, $\th$ be a positive constant $\le\min\left\{\th_\ep,\f18\de\be_\ep^{-2}\left(c_{n,\k r}+\La_r\right)^{-1}\right\}$, and $\th'$ be a positive constant $\le\min\left\{\th_\ep,\sqrt{\de c_{n,\k r}^{-1}\be_\ep^{-1}}\right\}$, then from \eqref{nabirSi*}\eqref{Hessbi*} it follows that
\begin{equation}\aligned
\fint_{B_{2\th'\th r}(x_{i,j})}\left|\na\mathbf{b}_i-\na \r_{\Si}\right|^2\le\de,
\endaligned
\end{equation}
and
\begin{equation}\aligned
\fint_{B_{2\th'\th r}(x_{i,j})}\left|\mathrm{Hess}_{\mathbf{b}_i}\right|^2\le\f{\de}{\th^2(\th')^2r^2}.
\endaligned
\end{equation}
This completes the proof.
\end{proof}

Let $q$ be a fixed point in $B_R(p)$, for any point $x\in B_R(p)$ let $\g_{x,q}$ denote a normalized minimizing geodesic in $B_{2R}(p)$ with $\g_{x,q}(0)=x$, $\g_{x,q}(d(q,x))=q$ and $|\dot{\g}_{x,q}|=1$.
\begin{lemma}\label{bzuHessu}
Let $B_{2r}(z)\subset B_{R}(p)$, $q$ be a point in $B_R(p)$ with $d(q,z)\ge2\th r$ for $0<\th<1$. If $u$ is a smooth function on $B_{2\th r}(z)$, then for any $t\in(0,\th r)$ we have
\begin{equation}\aligned
&\int_{x\in B_{\th r}(z)}\left|(\r_\Si\circ\g_{x,q})'(t)-\f{\r_\Si(\g_{x,q}(\th r))-\r_\Si(x)}{\th r}\right|\\
\le&c_{n,\k r}\int_{B_{2\th r}(z)}\left(\left|\na(\r_\Si-u)\right|+\th r\left|\mathrm{Hess}_{u}\right|\right),
\endaligned
\end{equation}
where $c_{n,\k r}$ is a general positive constant depending on $n,\k r$.
\end{lemma}
\begin{proof}
For any $L^1$-function $f\ge0$ on $B_{2r}(z)\subset B_{R}(p)$, $q\in B_{R}(p)$, $d(q,z)\ge2\th r$, $0<\th<1$, $0<t<\th r$, we claim
\begin{equation}\aligned\label{fgqxt}
\int_{x\in B_{\th r}(z)}f(\g_{x,q}(t))\le c_{n,\k r}\int_{B_{\th r+t}(z)}f.
\endaligned
\end{equation}
Let us prove \eqref{fgqxt} using the Laplacian comparison theorem (see also \cite{H0} for instance).
For any fixed $q\in B_{R}(p)$, let $\mathcal{C}_q$ denote the cut locus of the distance function from $q$.
Let $$U_{q,z,t,\th r}=\left\{\g_{x,q}(t)|\,x\in \mathcal{C}_q\cap B_{\th r}(z)\right\}.$$
Then for each $y\in U_{q,z,t,\th r}$, there is a unique $x\in \mathcal{C}_q\cap B_{\th r}(z)$ with $y=\g_{x,q}(t)$.
Let $dl^2+g_{\a\be}(l,\vartheta)d\vartheta_\a d\vartheta_\be$ denote the metric of $N$ in the polar coordinate w.r.t. $q$ outside $\mathcal{C}_q$, where $\vartheta=(\vartheta_1,\cdots,\vartheta_n)$ satisfies $|\vartheta|=1$.
Let $J(l,\vartheta)=\sqrt{\det(g_{\a\be}(l,\vartheta))}$.
By the Laplacian comparison theorem, we have
$$J(l+t,\vartheta)\le c_{n,\k r}J(l,\vartheta)\qquad\qquad \mathrm{for\ any}\ \  0<t<\th r\le l,$$
where $c_{n,\k r}$ is a positive constant depending on $n,\k r$.
Hence,
\begin{equation}\aligned
\int_{x\in B_{\th r}(z)}f(\g_{x,q}(t))=&\int_{\p B_1(0^{n+1})} \left(\int_{2\th r}^\infty f\left(\mathrm{exp}_{q}((l'-t)\vartheta)\right)\chi_{_{B_{\th r}(z)}}J(l',\vartheta)dl'\right)d\vartheta\\
=&\int_{\p B_1(0^{n+1})} \left(\int_{\th r}^\infty f\left(\mathrm{exp}_{q}(l\vartheta)\right)\chi_{_{U_{q,z,t,\th r}}}J(l+t,\vartheta)dl\right)d\vartheta\\
\le c_{n,\k r}\int_{\p B_1(0^{n+1})}& \left(\int_0^\infty f\left(\mathrm{exp}_{q}(l\vartheta)\right)\chi_{_{U_{q,z,t,\th r}}}J(l,\vartheta)dl\right)d\vartheta=c_{n,\k r}\int_{U_{q,z,t,\th r}} f.
\endaligned
\end{equation}
Combining $U_{q,z,t,\th r}\subset B_{\th r+t}(z)$, we get the claim \eqref{fgqxt}.

From the proof of Lemma 1.14 in \cite{C}, for any $t\in(0,\th r)$, $d(q,z)\ge2\th r$ we have
\begin{equation}\aligned
&\int_{x\in B_{\th r}(z)}\left|(\r_\Si\circ\g_{x,q})'(t)-\f{\r_\Si(\g_{x,q}(\th r))-\r_\Si(x)}{\th r}\right|\\
\le&\int_{x\in B_{\th r}(z)}\left|((\r_\Si-u)\circ\g_{x,q})'(t)\right|+\int_{x\in B_{\th r}(z)}\left|\int_0^{\th r}\f{((\r_\Si-u)\circ\g_{x,q})'(t)dt}{\th r}\right|\\
&+\int_{x\in B_{\th r}(z)}\left|(u\circ\g_{x,q})'(t)-\f{u(\g_{x,q}(\th r))-u(x)}{\th r}\right|\\
\le&2\sup_{t\in(0,\th r)}\int_{x\in B_{\th r}(z)}\left|\na(\r_\Si-u)\right|(\g_{x,q}(t))+2\int_{x\in B_{\th r}(z)}\left(\int_{0}^{\th r}\left|(u\circ\g_{x,q})''(t)\right|dt\right)\\
\le&2\sup_{t\in(0,\th r)}\int_{x\in B_{\th r}(z)}\left|\na(\r_\Si-u)\right|(\g_{x,q}(t))+2\th r\sup_{t\in(0,\th r)}\int_{x\in B_{\th r}(z)}\left|\mathrm{Hess}_{u}\right|(\g_{x,q}(t)).
\endaligned
\end{equation}
Combining \eqref{fgqxt}, we complete the proof.
\end{proof}

Now let us prove an angle estimate for the distance function $\r_\Si$ from $\Si$ as follows.
\begin{theorem}\label{bzth*}
Let $N$ be an $(n+1)$-dimensional complete Riemannian manifold with $\mathrm{Ric}_N\ge-n\k^2$ on $B_{R+R'}(p)$ for some constants $\k\ge0$ and $R'\ge R>0$.
Let $\Si$ be the support of an $n$-rectifiable varifold in $B_{2R}(p)$ with bounded mean curvature satisfying \eqref{DeNSiTh}.
Let $\La_r$ be the constant defined in \eqref{SiLa}.
For any $q\in B_R(p)$, $B_{2r}(z)\subset B_R(p)$ and $\ep>0$, there is a constant $\th_*$ depending on $n,\k R,\La_r,\ep$ such that
\begin{equation*}\aligned
\fint_{x\in B_{r}(z)}\left|(\r_\Si\circ\g_{x,q})'(t)-\f{\r_\Si(\g_{x,q}(\th r))-\r_\Si(x)}{\th r}\right|<2\f{\mathcal{H}^{n+1}\left(B_r(z)\cap B_{3\th r}(q)\right)}{\mathcal{H}^{n+1}(B_r(z))}+\ep
\endaligned
\end{equation*}
for any $0<\th\le \th_*$ and $0<t\le\th r$.
\end{theorem}
\begin{proof}
The idea of the proof comes from Colding (see Proposition 1.32 in \cite{C}).
From Lemma \ref{Decomb}, for any $\de >0$ there is a constant $\th_*=\th_{n,\k r,\La_r,\de }$ depending on $n,\k r,\La_r,\de $ such that
for each $0<\th\le\th_*$, there exist finitely many balls $B_{\th r}(p_i )\subset B_r(z)$ for $i=1,\cdots,N_*$ and harmonic functions $\mathbf{b}_i$ with
\begin{equation}\aligned\label{thrpide}
\mathcal{H}^{n+1}\left(\cup_{1\le i\le N_*}B_{\th r}(p_i )\right)\ge(1-\de )\mathcal{H}^{n+1}\left(B_{r}(z)\right),
\endaligned
\end{equation}
\begin{equation}\aligned\label{2thrbirside}
\fint_{B_{2\th r}(p_i )}\left|\na\mathbf{b}_i-\na \r_\Si\right|^2\le \de ,
\endaligned
\end{equation}
and
\begin{equation}\aligned\label{2thrHbirde}
\fint_{B_{2\th r}(p_i )}\left|\mathrm{Hess}_{\mathbf{b}_i}\right|^2\le \f{\de }{\th^2r^2}.
\endaligned
\end{equation}
By covering lemma (see \cite{St} for instance), we can assume that $B_{\th r/5}(p_i )\cap B_{\th r/5}(p_j )=\emptyset$ for all $i\neq j$.
From Lemma \ref{bzuHessu}, Cauchy inequality and \eqref{2thrbirside}\eqref{2thrHbirde}, for $0<t\le\th r$ and $d(q,p_i )\ge2\th r$ we have
\begin{equation}\aligned\label{2thrrSiC}
&\f1{\mathcal{H}^{n+1}(B_{2\th r}(p_i ))}\int_{x\in B_{\th r}(p_i )}\left|(\r_\Si\circ\g_{x,q})'(t)-\f{\r_\Si(\g_{x,q}(\th r))-\r_\Si(x)}{\th r}\right|\\
\le &c_{n,\k r}\fint_{B_{2\th r}(p_i )}\left(\left|\na(\r_\Si-\mathbf{b}_i)\right|+\th r\left|\mathrm{Hess}_{\mathbf{b}_i}\right|\right)\\
\le&c_{n,\k r}\left(\left(\fint_{B_{2\th r}(p_i )}\left|\na(\r_\Si-\mathbf{b}_i)\right|^2\right)^{1/2}
+\th r\left(\fint_{B_{2\th r}(p_i )}\left|\mathrm{Hess}_{\mathbf{b}_i}\right|^2\right)^{1/2}\right)\\
\le& c_{n,\k r}\left(\sqrt{\de }+\sqrt{\de }\right),
\endaligned
\end{equation}
where $c_{n,\k r}$ is a general constant depending on $n,\k r$.
We define an index set $\mathcal{I}$ by
$$\mathcal{I}=\{1\le i\le N_*|\, d(q,p_i )\ge2\th r\}.$$
From \eqref{thrpide}\eqref{2thrrSiC}, we have
\begin{equation}\aligned
&\int_{x\in B_{r}(z)}\left|(\r_\Si\circ\g_{x,q})'(t)-\f{\r_\Si(\g_{x,q}(\th r))-\r_\Si(x)}{\th r}\right|\\
\le& 2\mathcal{H}^{n+1}\left(B_r(z)\cap B_{3\th r}(q)\right)+2\mathcal{H}^{n+1}\left(B_{r}(z)\setminus\cup_{1\le i\le N_*}B_{\th r}(p_i )\right)\\
&+\int_{x\in\cup_{i\in \mathcal{I}} B_{\th r}(p_i )}\left|(\r_\Si\circ\g_{x,q})'(t)-\f{\r_\Si(\g_{x,q}(\th r))-\r_\Si(x)}{\th r}\right|\\
\le& 2\mathcal{H}^{n+1}\left(B_r(z)\cap B_{3\th r}(q)\right)+2\de \mathcal{H}^{n+1}\left(B_r(z)\right)\\
& +\sum_{i\in \mathcal{I}}\int_{x\in B_{\th r}(p_i )}\left|(\r_\Si\circ\g_{x,q})'(t)-\f{\r_\Si(\g_{x,q}(\th r))-\r_\Si(x)}{\th r}\right|\\
\le& 2\mathcal{H}^{n+1}\left(B_r(z)\cap B_{3\th r}(q)\right)+2\de \mathcal{H}^{n+1}\left(B_r(z)\right)+2c_{n,\k r}\sqrt{\de }\sum_{i\in \mathcal{I}}\mathcal{H}^{n+1}(B_{2\th r}(p_i )).
\endaligned
\end{equation}
With Bishop-Gromov comparison theorem,
\begin{equation}\aligned
\sum_{i\in \mathcal{I}}\mathcal{H}^{n+1}(B_{2\th r}(p_i ))\le c_{n,\k r}\sum_{i\in \mathcal{I}}\mathcal{H}^{n+1}\left(B_{\th r/5}(p_i )\right)\le c_{n,\k r}\mathcal{H}^{n+1}(B_{r}(z)),
\endaligned
\end{equation}
which completes the proof.
\end{proof}

\section{Limits of the distance functions from minimal hypersurfaces}

Let $N_i$ be a sequence of $(n+1)$-dimensional complete Riemannian manifolds with $\mathrm{Ric}_{N_i}\ge-n\k^2$ on $B_{R+R'}(p_i)\subset N_i$ for some constants $\k\ge0$ and $R'\ge R>0$.
Suppose that $B_{R+R'}(p_i)$ converges to a metric ball $B_{R+R'}(p_\infty)$ in the Gromov-Hausdorff sense.
Namely, for each integer $i\ge1$, there is a sequence of $\ep_i$-Gromov-Hausdorff approximations $\Phi_i:\, \overline{B_{R+R'}(p_i)}\to\overline{B_{R+R'}(p_\infty)}$ for some sequence $\ep_i\to0$. Let $\nu$ be a unique Radon measure on $B_{R+R'}(p_\infty)$ defined as \eqref{nuinfty}.
For any $x,q\in B_R(p_\infty)$, let $\g_{x,q}$ denote a minimal geodesic segment from $x$ to $q$ in $B_{2R}(p_\infty)$ with $|\dot{\g}_{x,q}|=1$, $\g_{x,q}(0)=x$ and $\g_{x,q}(\r_q(x))=q$. 
For simplicity, let $\r_q$ denote the distance function $d(q,\cdot)$ for each $q\in B_R(p_\infty)$ or $q\in B_R(p_i)$.

For each integer $i\ge1$, let $f_i$ be a Lipschitz function on $B_{R}(p_i)$ satisfying
$$\limsup_{i\rightarrow\infty}\sup_{B_{R}(p_i)}\left(|f_i|+\mathrm{Lip} f_i\right)<\infty.$$
For a function $f_\infty$ on $B_{R}(p_\infty)$, we say $f_i\rightarrow f_\infty$ on $B_{R}(p_\infty)$ if $f_i(x_i)\rightarrow f_\infty(x)$ for any $x\in B_{R}(p_\infty)$ and any sequence $x_i\in B_{R}(p_i)$ with $x_i\rightarrow x$. In particular, $f_\infty$ is Lipschitz with $\mathbf{Lip} f_\infty\le\limsup_{i\rightarrow\infty}\mathbf{Lip} f_i$.
According to Honda \cite{H1}, we further say $df_i\to df_\infty$ on $B_R(p_\infty)$ if for every $\de>0$, every $x_\infty\in B_R(p_\infty)$, every $z_\infty\in B_R(p_\infty)$, every sequence $B_R(p_i)\ni x_i\to x_\infty$, and every sequence $B_R(p_i)\ni z_i\to z_\infty$ there exists a constant $r>0$ such that
\begin{equation}\aligned
\limsup_{i\rightarrow\infty}\left|\fint_{B_{t}(x_i)}\lan df_i,d\r_{z_i}\ran-\fint_{B_{t}(x_\infty)}\lan df_\infty,d\r_{z_\infty}\ran d\nu\right|<\de
\endaligned
\end{equation}
and
\begin{equation}\aligned
\limsup_{i\rightarrow\infty}\fint_{B_{t}(x_i)}|df_i|^2\le\fint_{B_{t}(x_\infty)}|df_\infty|^2d\nu+\de
\endaligned
\end{equation}
for any $t\in(0,r]$.
We denote $(f_i,df_i)\to(f_\infty,df_\infty)$ on $B_R(p_\infty)$ if both $f_i\to f_\infty$ and $df_i\to df_\infty$ on $B_R(p_\infty)$.
Moreover, Honda gave a more general definition beyond the Lipschitz condition \cite{H1}.

Let $\phi_\infty$ be a Lipschitz function on $B_{R}(p_\infty)$. From Lemma 10.7 in \cite{Ch} by Cheeger, there is a sequence of Lipschitz functions $\phi_i$ on $B_{R}(p_i)$ such that $\phi_i\to \phi_\infty$ with $\mathbf{Lip} \phi_i\le\mathbf{Lip} \phi_\infty$.
Furthermore, from Theorem 4.2 in \cite{H1}, $\phi_i$ can be chosen so that $(\phi_i,d\phi_i)\to(\phi_\infty,d\phi_\infty)$ on $B_R(p_\infty)$.
In Theorem 1.1 of \cite{H1}, Honda proved the following theorem.
\begin{theorem}\label{Honda}
For each $B_t(x_\infty)\subset B_R(p_\infty)$, each sequence $B_R(p_i)\ni x_i\to x_\infty$, let $f_i$ be a Lipschitz function on $B_{R}(p_i)$ with $\limsup_{i\rightarrow\infty}\mathbf{Lip} f_i<\infty$ such that $(f_i,df_i)\to(f_\infty,df_\infty)$ on $B_R(p_\infty)$ for some Lipschitz function $f_\infty$ on $B_R(p_\infty)$. Then there holds
\begin{equation}\aligned
\lim_{i\rightarrow\infty}\fint_{B_{t}(x_i)}\lan df_i,d\phi_i\ran=\fint_{B_{t}(x_\infty)}\lan df_\infty,d\phi_\infty\ran d\nu.
\endaligned
\end{equation}
\end{theorem}

Let $M_i$ be a closed set in $B_{R+R'}(p_i)$ for each $i$.
Suppose that $\Phi_i(M_i)$ converges to a closed set $M_\infty$ in $\overline{B_{R+R'}(p_\infty)}$ in the Hausdorff sense.
Let $\r_{M_\infty}$ denote the distance function from $M_\infty$ on $\overline{B_{R}(p_\infty)}$, i.e., $\r_{M_\infty}(x)=\inf_{y\in M_\infty}d(x,y)$.
From Lemma \ref{c0} in the Appendix I, $\r_{M_i}\to\r_{M_\infty}$ on $B_R(p_\infty)$. Moreover, the convergence is uniform in the following sense.
For any $\ep>0$, there is an $\ep$-net $\{y_1,\cdots,y_{m_\ep}\}$ such that $B_R(p_\infty)\subset\cup_{j=1}^{m_\ep}B_\ep(y_j)$.
For each $j$, let $y_{i,j}\in B_R(p_i)$ be a sequence converging to $y_j$. Then there is an integer $i_\ep>0$ such that $d(y_j,\Phi_i(y_{i,j}))<\ep$ and $|\r_{M_i}(y_{i,j})-\r_{M_\infty}(y_j)|<\ep$ for all $i\ge i_\ep$ and $j=1,\cdots,m_\ep$.
For a point $x\in B_R(p_\infty)$, there is an integer $j$ satisfying $x\in B_\ep(y_j)$.
From $\mathbf{Lip} \r_{M_i},\mathbf{Lip} \r_{M_\infty}\le1$ and the definition of $\Phi_i$, we have
\begin{equation}\aligned\label{rMixiMinfx}
|\r_{M_i}(x_i)-\r_{M_\infty}(x)|\le&|\r_{M_i}(x_i)-\r_{M_i}(y_{i,j})|+|\r_{M_i}(y_{i,j})-\r_{M_\infty}(x)|\\
\le&d(x_i,y_{i,j})+|\r_{M_i}(y_{i,j})-\r_{M_\infty}(y_j)|
+|\r_{M_\infty}(y_j)-\r_{M_\infty}(x)|\\
\le&d(\Phi_i(x_i),\Phi_i(y_{i,j}))+2\ep_i+\ep+d(y_j,x)\\
\le&d(\Phi_i(x_i),x)+d(x,y_j)+d(y_j,\Phi_i(y_{i,j}))+2\ep_i+2\ep\\
\le&d(x,\Phi_i(x_i))+2\ep_i+4\ep
\endaligned
\end{equation}
for each $i\ge i_\de$.

We further assume that $M_i$ is the support of a rectifiable stationary $n$-varifold in $B_{R+R'}(p_i)$ for each integer $i\ge1$ such that $M_i\cap B_{R'}(p_i)\neq\emptyset$ and
\begin{equation}\aligned
\limsup_{i\rightarrow\infty}\f{\mathcal{H}^n(M_i\cap B_{R}(p_i))}{\mathcal{H}^{n+1}(B_R(p_i))}\le\La
\endaligned
\end{equation}
for some constant $\La>0$. Then from the Laplacian comparison for $\r_{M_i}$ by Heintze-Karcher \cite{HeK},
we obtain
\begin{equation}\label{DeNirMi}
\De_{N_i}\r_{M_i}\le n\k\tanh\left(\k\r_{M_i}\right)\qquad \mathrm{on}\ B_{R}(p_i)\setminus M_i
\end{equation}
in the distribution sense (see also Lemma 7.1 in \cite{D1}), where $\De_{N_i}$ denotes the Laplacian of $N_i$.
\begin{theorem}\label{c1}
For each Lipschitz function $\phi$ on $B_R(p_\infty)$, there is a sequence of Lipschitz functions $\phi_i$ on $B_R(p_i)$ satisfying $(\phi_i,d\phi_i)\to(\phi,d\phi)$ on $B_R(p_\infty)$ and $\limsup_{i\rightarrow\infty}\mathbf{Lip}\,\phi_i\le\mathbf{Lip}\,\phi$ such that
\begin{equation}\aligned
\lim_{i\rightarrow\infty}\fint_{B_R(p_{i})}\left\lan d\r_{M_{i}},d\phi_{i}\right\ran=\fint_{B_R(p_\infty)}\left\lan d\r_{M_\infty},d\phi\right\ran d\nu.
\endaligned
\end{equation}
If we further suppose that $\phi$ has compact support in $B_R(p_\infty)\setminus M_\infty$, then we can require that the function $\phi_i$ has compact support in $B_R(p_i)\setminus M_i$ for each $i$.
\end{theorem}
\begin{proof}
Let $\nu$ be the normalized measure obtained from $B_R(p_i)$ defined in \eqref{nuinfty}.
For any fixed $q\in B_R(p_\infty)$, there is a subset $Z_q$ in the set of restricted cut points of $q$ in $B_R(p_\infty)$ (see \cite{CCo1}) such that $\nu(B_{R}(p_\infty)\setminus Z_q)=0$, and two Lipschitz functions $\r_{M_\infty}$, $\r_q$ are differentiable on $Z_q$. We define a bounded measurable function $F(x,t)=F_t(x)$ for every $t\in[0,d(x,q)]$ and every $x\in B_R(p_\infty)$ with $|F_t|\le2$ a.e. by
\begin{equation}\aligned
F_t(x)=\lan d\r_{M_\infty},d\r_q\ran(x)-t^{-1}(\r_{M_\infty}(\g_{x,q}(t))-\r_{M_\infty}(x)).
\endaligned
\end{equation}
From Theorem 3.3 of \cite{H1}, we get $\lim_{t\rightarrow0}F_t(x)=0$ for each $x\in Z_q$.
From Lusin's theorem and Egoroff's theorem,
for any $\ep>0$, any $B_{2r}(z)\subset B_R(p_\infty)$, there exist a constant $0<\th<<\min\{\ep,r\}$, and a Borel subset $Z_q$ with $Z_q\subset Z_\ep\subset B_{r}(z)$ and $\nu(B_{r}(z)\setminus Z_\ep)<\ep\nu(B_{r}(z))$, such that
$|F_t(x)|<\ep$ for every $t\in(0,\th]$ and every $x\in Z_\ep$ with $\r_q(x)\ge2\th$.
Then for each $t\in(0,\th]$, we have
\begin{equation}\aligned\label{trMxqx}
\int_{B_{r}(z)}|F_t|d\nu
\le\int_{Z_\ep}|F_t|d\nu+2\nu(B_{r}(z)\setminus Z_\ep)
\le\ep\nu(Z_\ep)+2\ep\nu(B_{r}(z))\le3\ep\nu(B_{r}(z)).
\endaligned
\end{equation}
Let $q_i\in B_R(p_i)$ with $q_i\rightarrow q$, and $z_i\in B_R(p_i)$ with $z_i\rightarrow z$.
From Theorem \ref{bzth*}, there is a constant $\th_*=\th_{n,\k R,\La,\ep}$ depending on $n,\k R,\La,\ep$ such that
\begin{equation}\aligned\label{trMixqix}
\fint_{x_i\in B_{r}(z_i)}\left|\lan d\r_{M_i},d\r_{q_i}\ran(x_i)-\f{\r_{M_i}(\g_{x_i,q_i}(t))-\r_{M_i}(x_i)}{t}\right|
<2\f{\mathcal{H}^{n+1}\left(B_r(z_i)\cap B_{3t}(q_i)\right)}{\mathcal{H}^{n+1}(B_r(z_i))}+\ep
\endaligned
\end{equation}
for any $0<t\le \th_*r$.

From \eqref{rMixiMinfx}, there is an integer $i_t$ depending on $t$ such that for any $x\in B_r(z)$ and $i\ge i_t$
\begin{equation}\aligned
\left|\r_{M_i}(x_i)-\r_{M_\infty}(x)\right|<d(x,\Phi_i(x_i))+\f{t\ep}4.
\endaligned
\end{equation}
For the suitable small $t>0$, $x\in B_r(z)\cap Z_q$ and $i\ge i_t$, we have
\begin{equation}\aligned
\left|\r_{M_i}(\g_{x_i,q_i}(t))-\r_{M_\infty}(\g_{x,q}(t))\right|<d(\g_{x,q}(t),\Phi_i(\g_{x_i,q_i}(t)))+\f{t\ep}4<d(x,\Phi_i(x_i))+\f{t\ep}2.
\endaligned
\end{equation}
With the property of Radon measure $\nu$ defined in \eqref{nuinfty}, we have
\begin{equation}\aligned\label{BRzirMigt0MINF}
\bigg|\f1{\mathcal{H}^{n+1}(B_R(z_i))}\int_{x_i\in B_{r}(z_i)}&\left(\r_{M_i}(\g_{x_i,q_i}(t))-\r_{M_i}(x_i)\right)\\
&-\int_{x\in B_r(z)}\left(\r_{M_\infty}(\g_{x,q}(t))-\r_{M_\infty}(x)\right)d\nu\bigg|<t\ep
\endaligned
\end{equation}
for all $i\ge i_t$. Hence, for the suitable small $\th_*=\th_{n,\k R,\La,\ep}>0$, from \eqref{nuinfty} again and \eqref{BRzirMigt0MINF} there is an integer $i_t'\ge i_t$ depending on $t$ such that
\begin{equation}\aligned\label{gxqiq}
\left|\fint_{x_i\in B_{r}(z_i)}\left(\r_{M_i}(\g_{x_i,q_i}(t))-\r_{M_i}(x_i)\right)-\fint_{x\in B_r(z)}\left(\r_{M_\infty}(\g_{x,q}(t))-\r_{M_\infty}(x)\right)d\nu\right|<t\ep
\endaligned
\end{equation}
for each $t\in(0,\min\{\th,\th_* r\})$ and each $i\ge i_t'$.
Combining \eqref{trMxqx}\eqref{trMixqix}\eqref{gxqiq}, for $i\ge i_t'$ we have
\begin{equation}\aligned
&\left|\fint_{B_{r}(z_i)}\lan d\r_{M_i},d\r_{q_i}\ran-\int_{B_{r}(z)}\lan d\r_{M_\infty},d\r_q\ran d\nu\right|\\
\le&\fint_{x_i\in B_{r}(z_i)}\left|\lan d\r_{M_i},d\r_{q_i}\ran(x_i)-\f{\r_{M_i}(\g_{x_i,q_i}(t))-\r_{M_i}(x_i)}{t}\right|+\fint_{B_{r}(z)}|F_t|d\nu\\
+&\f1t\left|\fint_{x_i\in B_{r}(z_i)}\left(\r_{M_i}(\g_{x_i,q_i}(t))-\r_{M_i}(x_i)\right)-\fint_{x\in B_{r}(z)}\left(\r_{M_\infty}(\g_{x,q}(t))-\r_{M_\infty}(x)\right)d\nu\right|\\
\le&2\f{\mathcal{H}^{n+1}\left(B_{r}(z_i)\cap B_{3t}(q_i)\right)}{\mathcal{H}^{n+1}(B_r(z_i))}+\ep+3\ep+\ep.
\endaligned
\end{equation}
With \eqref{nuinfty}, letting $i\rightarrow\infty$ in the above inequality implies
\begin{equation}\aligned\label{rMirqirt5ep}
\lim_{i\rightarrow\infty}\left|\fint_{B_{r}(z_i)}\lan d\r_{M_i},d\r_{q_i}\ran-\fint_{B_r(z)}\lan d\r_{M_\infty},d\r_q\ran d\nu\right|
\le2\f{\nu\left(B_r(z)\cap B_{3t}(q)\right)}{\nu(B_r(z))}+5\ep.
\endaligned
\end{equation}
From Bishop-Gromov volume comparison, we clearly have $\lim_{t\to0}\nu\left(B_{3t}(q)\right)=0$.
Since $t,\ep$ can be chosen arbitrarily small, from \eqref{rMirqirt5ep} it follows that
\begin{equation}\aligned\label{irMirM}
\lim_{i\rightarrow\infty}\fint_{B_{r}(z_i)}\lan d\r_{M_i},d\r_{q_i}\ran=\fint_{B_{r}(z)}\lan d\r_{M_\infty},d\r_q\ran d\nu.
\endaligned
\end{equation}
Combining \eqref{rMixiMinfx}\eqref{irMirM}, we get $(\r_{M_i},d\r_{M_i})\to(\r_{M_\infty},d\r_{M_\infty})$ on $B_R(p_\infty)$. For each Lipschitz function $\phi$ on $B_R(p_\infty)$, from Lemma 10.7 in \cite{Ch} and Theorem 4.2 in \cite{H1}, there is a sequence of Lipschitz functions $\phi_i$ on $B_R(p_i)$
satisfying $(\phi_i,d\phi_i)\to(\phi,d\phi)$ on $B_R(p_\infty)$ and $\limsup_{i\rightarrow\infty}\mathbf{Lip}\,\phi_i\le\mathbf{Lip}\,\phi$.
From Theorem \ref{Honda}, we get
\begin{equation}\aligned
\lim_{i\rightarrow\infty}\fint_{B_R(p_{i})}\left\lan d\r_{M_{i}},d\phi_{i}\right\ran=\fint_{B_R(p_\infty)}\left\lan d\r_{M_\infty},d\phi\right\ran d\nu.
\endaligned
\end{equation}

We further suppose that $\phi$ has the compact support in $B_R(p_\infty)\setminus M_\infty$. Then there is a small constant $\de>0$ such that spt$\phi\cap B_{2\de}(M_\infty)=\emptyset$. From $(\phi_i,d\phi_i)\to(\phi,d\phi)$ on $B_R(p_\infty)$, we have
\begin{equation}\aligned\label{phiiphidd}
\lim_{i\rightarrow\infty}\left(\sup_{B_R(p_i)\cap B_{3\de/2}(M_i)}|\phi_i|+\f1{\mathcal{H}^{n+1}(B_R(p_i))}\int_{B_R(p_i)\cap B_\de(M_i)}|d\phi_i|\right)=0.
\endaligned
\end{equation}
Let $\e_{i,\de}$ be a Lipschitz function on $B_R(p_i)$ defined by $\e_{i,\de}=\f1{\de}\r_{M_i}$ on $B_\de(M_i)\cap B_R(p_i)$, and $\e_{i,\ep}=1$ on $B_R(p_i)\setminus B_\de(M_i)$. Set $\phi_i^*=\e_{i,\de}\phi_i$, then $\phi_i^*\in \mathrm{Lip}_0(B_R(p_i)\setminus M_i)$. $\phi_i\to\phi$ on $B_R(p_\infty)$ and spt$\phi\cap B_{2\de}(M_\infty)=\emptyset$ imply $\phi_i^*\to\phi$ on $B_R(p_\infty)$. Moreover, with \eqref{phiiphidd} it follows that 
\begin{equation}\aligned\label{phii***phii}
\limsup_{i\rightarrow\infty}\mathbf{Lip}\,\phi_i^*\le\limsup_{i\rightarrow\infty}\mathbf{Lip}\,\phi_i\le\mathbf{Lip}\,\phi.
\endaligned
\end{equation}
From
\begin{equation}\aligned
\left|d\phi_i^*-d\phi_i\right|=\left|(\e_{i,\de}-1)d\phi_i+\phi_id\e_{i,\de}\right|\le\chi_{_{B_R(p_i)\cap B_\de(M_i)}}\left(|d\phi_i|+\f1\de|\phi_i|\right),
\endaligned
\end{equation}
and \eqref{phiiphidd}, we get
\begin{equation}\aligned
\limsup_{i\rightarrow\infty}\fint_{B_R(p_{i})}\left|d\phi_i^*-d\phi_i\right|\le\limsup_{i\rightarrow\infty}\f1{\mathcal{H}^{n+1}(B_R(p_i))}\int_{B_R(p_i)\cap B_\de(M_i)}\left(|d\phi_i|+\f1\de|\phi_i|\right)=0.
\endaligned
\end{equation}
So we get $(\phi_i^*,d\phi_i^*)\to(\phi,d\phi)$ on $B_R(p_\infty)$. With \eqref{phii***phii},
we complete the proof.
\end{proof}
From \eqref{rMixiMinfx}\eqref{DeNirMi} and Theorem \ref{c1}, we immediately have the following corollary.
\begin{corollary}
Let $\phi$ be a nonnegative Lipschitz function on $B_R(p_\infty)$ with compact support in $B_R(p_\infty)\setminus M_\infty$, then 
\begin{equation}\aligned\label{drMdphige0}
\int_{B_R(p_\infty)}\left\lan d\r_{M_\infty},d\phi\right\ran d\nu\ge-n\k\int_{B_R(p_\infty)}\phi\tanh\left(\k\r_{M_\infty}\right)d\nu.
\endaligned
\end{equation}
\end{corollary}

\section{Limiting cones from minimal hypersurfaces}

Let $R_i\ge0$ be a sequence with $R_i\rightarrow\infty$ as $i\rightarrow\infty$.
Let $B_{R_i}(p_i)$ be a sequence of $(n+1)$-dimensional smooth geodesic balls with Ricci curvature $\ge-nR_i^{-2}$ such that
$(B_{R_i}(p_i),p_i)$ converges to a metric cone $(\mathbf{C},\mathbf{o})$ in the pointed Gromov-Hausdorff sense.
Suppose $\liminf_{i\to\infty}\mathcal{H}^{n+1}(B_1(p_i))>0$, and $\mathbf{C}$ splits off a Euclidean factor $\R^{n-k}$ isometrically for some integer $1\le k\le n$.
Then the measure $\nu$ defined in \eqref{nuinfty} is just a multiple of the Hausdorff measure $\mathcal{H}^{n+1}$ and the volume convergence
\begin{equation}
\mathcal{H}^{n+1}(B_1(p_\infty))=\lim_{i\rightarrow\infty}\mathcal{H}^{n+1}(B_1(p_i))
\end{equation}
holds from Colding \cite{C}, Cheeger-Colding \cite{CCo1}.
Moreover, there is a $k$-dimensional compact metric space $X$ such that $\mathbf{C}=CX\times\R^{n-k}$, where $CX$ is a metric cone with the vertex $o$ of the cross section $X$.

Let $B_r(\mathbf{o})$ be the ball of radius $r$ and centered at $\mathbf{o}$ in $\mathbf{C}$,
and $B_r(q)$ be the ball of radius $r$ and centered at $q\in CX$ in $CX$.
For any $p_1,p_2\in CX\times\{z\}\subset\mathbf{C}$ for some $z\in\R^{n-k}$, any minimizing geodesic joining two points $p_1,p_2$ must live in $CX\times\{z\}$.
From Cheeger-Colding \cite{CC}, $\p B_1(\mathbf{o})$ is connected with the diameter $\le\pi$ (see also Abresch-Gromoll theorem \cite{AG}),
which implies that $X$ is connected and the diameter of $X$ $\le\pi$.
From Cheeger-Colding \cite{CC,CCo1}, $\mathbf{C}$ is a volume cone, i.e.,
\begin{equation}\aligned\label{RnuRrnur}
R^{-n-1}\mathcal{H}^{n+1}(B_R(\mathbf{o}))=r^{-n-1}\mathcal{H}^{n+1}(B_r(\mathbf{o}))\qquad \mathrm{for\ any}\ 0<r<R.
\endaligned
\end{equation}
Since $\mathbf{C}$ splits off a Euclidean factor $\R^{n-k}$ isometrically, then the co-area formula 
\begin{equation}\aligned\label{coarean+1CXR}
\mathcal{H}^{n+1}(B_R(\mathbf{o}))=\int_{B_R(0^{n-k})}\mathcal{H}^{k+1}\left(B_{\sqrt{R^2-|x|^2}}(o)\right)dx
\endaligned
\end{equation}
holds, which follows that
\begin{equation}\aligned\label{RrHk}
R^{-k-1}\mathcal{H}^{k+1}(B_R(o))=r^{-k-1}\mathcal{H}^{k+1}(B_r(o))\qquad \mathrm{for\ any}\ 0<r<R.
\endaligned
\end{equation}
For any compact set $K\subset CX$ and any $t>0$, let $tK$ denote a subset in $CX$ obtained by scaling of $K$ with the factor $t$ such that $\mathcal{H}^{k+1}(tK)=t^{k+1}\mathcal{H}^{k+1}(K)$.
In particular, $tX=\p B_t(o)$.
With covering technique and Bishop-Gromov volume comparison, the cone $CX$ implies the co-area formula (see Proposition 7.6 in \cite{H2} for instance)
\begin{equation}\aligned\label{coareanuX}
\mathcal{H}^{k+1}(K)=\int_0^\infty \mathcal{H}^{k}(\p B_t(o)\cap K)dt.
\endaligned
\end{equation}

For any $\xi,\e\in X$, $t,\tau>0$, $y,z\in\R^{n-k}$, we denote $t\xi,\tau\e\in CX$, $(t\xi,y)\in CX\times\R^{n-k}$, $(\tau\e,z)\in CX\times\R^{n-k}$ for convenience.
Let $d_X(\cdot,\cdot)$, $d_{CX}(\cdot,\cdot)$, $d_{\mathbf{C}}(\cdot,\cdot)$ denote the distance functions on both of $X$, $CX$ and $\mathbf{C}$, respectively.
Then
\begin{equation}\aligned\label{disXCX}
d_{CX}(t\xi,\tau\e)^2=t^2+\tau^2-2t\tau\cos d_X(\xi,\e),
\endaligned
\end{equation}
and
\begin{equation}\aligned\label{disbfC}
d_{\mathbf{C}}((t\xi,y),(\tau\e,z))=&d_{CX}(t\xi,\tau\e)^2+|y-z|^2\\
=&t^2+\tau^2-2t\tau\cos d_X(\xi,\e)+|y-z|^2.
\endaligned
\end{equation}
For any point $x\in X$ and $r,t>0$, let $\mathscr{B}_r(tx)$ denote the metric ball in $tX=\p B_t(o)$ with the radius $r$ and centered at $tx$.
With Bishop-Gromov volume comparison, there is a constant $c_k\ge1$ depending only on $k$ such that (see \eqref{VolDVRr*} in the Appendix II or Proposition 7.9 in \cite{H2} by Honda for instance)
\begin{equation}\aligned\label{VolDVRr}
\mathcal{H}^{k}(\mathscr{B}_{R}(x))\le c_k\f{R^k}{r^k}\mathcal{H}^{k}(\mathscr{B}_{r}(x))\qquad \mathrm{for\ each}\ 0<r\le R\le1.
\endaligned
\end{equation}

Let $\z$ be a Lipschitz function on $\R^{n-k}$, and $f$ be a Lipschitz function on $CX$.
We have introduced the differential $d(f\z)$ in $\S 2$. Obviously, $d(f\z)$ gives the differential of $f$ on $CX$,
where the differential of $f$, denoted by $df$, is a $\mathcal{H}^{k+1}$-a.e. well-defined $L_\infty$ section of $T^*(CX)$.
Moreover, there is a Borel set $V_*\subset CX$ such that $\mathcal{H}^{k+1}(CX\setminus V_*)=0$ and the differential $df:\,V_*\rightarrow T^*(CX)$ satisfies $d(f\z)=\z df+fd\z$ and $\mathrm{Lip}\,(f\z)=|d(f\z)|$ on $V_*$.

Let $\phi$ be a Lipschitz function on $X$, and $\la$ be a Lipschitz function on $[0,\infty)$. Now we suppose
$$f(t\xi)=\la(t)\phi(\xi)\qquad \mathrm{for\ each}\ t\xi\in CX.$$
Let us define the differential of $\phi$ on the cross section $X$ using $df$.
More precisely,
we can use $d\phi$ denoting the differential of $\phi$, which is a $\mathcal{H}^k$-a.e. well-defined $L_\infty$ section of $T^*X$.
There are Borel sets $V_X\subset X$, $\La_X\subset[0,\infty)$ with $V\triangleq\{t\xi\in CX|\ t\in\La_X,\, \xi\in V_X\}\subset V_*$ such that $\mathcal{H}^k(X\setminus V_X)=0$, $\mathcal{H}^1([0,\infty)\setminus\La_X)=0$ and
the differential $df:\,V\rightarrow T^*(CX)$ satisfies $df=\la d\phi+\phi\la'dt$ and $\mathrm{Lip}\,f=|df|$ on $V$.
From \eqref{Lipftxi**} in the Appendix II, for every $t\xi\in V$ we have
\begin{equation}\aligned\label{Lipftxi*}
(\mathrm{Lip}\, f)^2(t\xi)
=\f{\la^2(t)}{t^2}(\mathrm{Lip}\, \phi)^2(\xi)+(\la'(t))^2\phi^2(\xi)=\f{\la^2(t)}{t^2}|d\phi|^2(\xi)+(\la'(t))^2\phi^2(\xi).
\endaligned
\end{equation}

Let $\tilde{\phi}$ be a Lipschitz function on $X$, $\tilde{\la}$ be a Lipschitz function on $[0,\infty)$, $\tilde{f}(t\xi)=\tilde{\la}(t)\tilde{\phi}(\xi)$ for each $t\xi\in CX$.
Let $\lan \cdot,\cdot\ran_{CX}, \lan \cdot,\cdot\ran_X$ denote the pointwise inner products on $CX$ and $X$, respectively.
From \eqref{quad} and \eqref{Lipftxi*}, we have
\begin{equation}\aligned\label{ftf}
\lan df,d\tilde{f}\ran_{CX}=\f{\la\tilde{\la}}{t^2}\lan d\phi,d\tilde{\phi}\ran_X+\la'\tilde{\la}'\phi\tilde{\phi}\qquad \quad\mathcal{H}^{k+1}-a.e.\ \mathrm{on}\ CX.
\endaligned
\end{equation}
With \eqref{quad}\eqref{disbfC}\eqref{ftf}, it follows that
\begin{equation}\aligned\label{dfz2=}
|d(f\z)|^2=\z^2|df|^2+f^2|d\z|^2=\f{\la^2\z^2}{t^2}|d\phi|^2+(\la')^2\z^2\phi^2+\la^2\phi^2|d\z|^2
\endaligned
\end{equation}
$\mathcal{H}^{n+1}$-a.e. on $\mathbf{C}$.
Let $\tilde{\z}$ be a Lipschitz function on $\R^{n-k}$.
Let $\lan \cdot,\cdot\ran_{\mathbf{C}}$, $\lan \cdot,\cdot\ran_{\R^{n-k}}$ denote the pointwise inner products on $\mathbf{C}$ and $\R^{n-k}$, respectively.
From \eqref{ftf}\eqref{dfz2=}, we have
\begin{equation}\aligned\label{fzCfz}
&\lan d(f\z),d(\tilde{f}\tilde{\z})\ran_{\mathbf{C}}=\z\tilde{\z}\lan df,d\tilde{f}\ran_{CX}+f\tilde{f}\lan d\z,d\tilde{\z}\ran_{\R^{n-k}}\\
=&\f{\la\tilde{\la}\z\tilde{\z}}{t^2}\lan d\phi,d\tilde{\phi}\ran_X+\la'\tilde{\la}'\phi\tilde{\phi}\z\tilde{\z}+\la\tilde{\la}\phi\tilde{\phi}\lan d\z,d\tilde{\z}\ran_{\R^{n-k}}\qquad \quad\mathcal{H}^{n+1}-a.e.\ \mathrm{on}\ \mathbf{C}.
\endaligned
\end{equation}

From the Poincar$\mathrm{\acute{e}}$ inequality \eqref{qPoincare}, up to a choice of the constant $c_k$, we can obtain (see \eqref{*Brphi*} in the Appendix II for instance)
\begin{equation}\aligned\label{*Brphi}
\int_{\mathscr{B}_r(x)}\left|\phi-\fint_{\mathscr{B}_r(x)}\phi\right|^q\le c_{k}r^q\int_{\mathscr{B}_{r}(x)}|d\phi|^q
\endaligned
\end{equation}
for any $q\ge1$, any Lipschitz function $\phi$ on $X$.
From Theorem 1 in \cite{HaK}, the inequality \eqref{*Brphi} holds for any $q\ge1$ and $0<r\le\mathrm{diam} X$ up to a choice of the constant $c_k$.
\begin{remark}
From Theorem 6.10 in \cite{AGS} by Ambrosio-Gigli-Savar$\mathrm{\acute{e}}$ and Theorem 3.22 in \cite{EKS} by Erbar-Kuwada-Sturm, $(CX,d_{CX},\mathcal{H}^{k+1})$ is a metric measure space satisfying $\mathrm{RCD}^*(0,k+1)$. In \cite{K}, Ketterer proved that $X$ satisfies $\mathrm{RCD}^*(k-1,k)$. From Theorem 6.2 in \cite{BS} by Bacher-Sturm, the generalized Bishop-Gromov volume growth inequality holds:
\begin{equation}\aligned\label{RatioXVol}
\f{\mathcal{H}^k(B_r(x))}{\mathcal{H}^k(B_R(x))}\ge\f{\int_0^r\sin^k(t\sqrt{l/k})dt}{\int_0^R\sin^k(t\sqrt{l/k})dt}
\endaligned
\end{equation}
for any $0<r <R\le\pi\sqrt{k/l}$ with $l\in(0,k-1]$. Letting $l\to0$ in \eqref{RatioXVol} implies the constant $c_k=1$ in \eqref{VolDVRr}.
Moreover, Rajala \cite{R} proved that $X$ supports the Poincar$\mathrm{\acute{e}}$ inequality \eqref{*Brphi} for $q=1$.
\end{remark}

With \eqref{VolDVRr}\eqref{*Brphi}, there holds the Sobolev inequality on $X$ for Lipschitz functions (see Theorem 5.1 and line 5, page 84 both in \cite{HaK} by Hajlasz-Koskela for instance).
In particular, for each Lipschitz function $\phi$ on $X$ with compact support in $\mathscr{B}_r(x)\subset X$ we have
\begin{equation}\aligned\label{SobX}
\left(\fint_{\mathscr{B}_r(x)}|\phi|^{\f{k}{k-1}}\right)^{\f{k-1}k}\le c_k'r\fint_{\mathscr{B}_{r}(x)}|d\phi|,
\endaligned
\end{equation}
where $c_k'$ is a constant depending only on $k$ and $\mathcal{H}^k(X)$.
With the famous De Giorgi-Nash-Moser iteration, we get the following mean value inequality on metric balls in $X$(refer to Theorem 3.2 in \cite{D0}).
\begin{lemma}\label{MVIsuper}
Let $\phi$ be a nonnegative Lipschitz function on $\mathscr{B}_{2r}(x)\subset X$ satisfying
$$\int_{\mathscr{B}_{2r}(x)}\lan d\phi,d\e\ran_X\ge0$$
for any nonnegative Lipschitz function $\e$ with support in $\mathscr{B}_{2r}(x)$.
Then $\phi$ satisfies the mean value inequality as follows:
\begin{equation}\aligned
\fint_{\mathscr{B}_r(x)}\phi^{k_*}\le c_k^*\phi^{k_*}(x),
\endaligned
\end{equation}
where $k_*\in(0,1],c_k^*>1$ are constants depending only on $k$ and $\mathcal{H}^k(X)$.
\end{lemma}

For each integer $i\ge1$, let $M_i$ be the support of rectifiable stationary $n$-varifold in $B_{R_i}(p_i)$ with $p_i\in M_i$ such that
\begin{equation}\aligned
\limsup_{i\rightarrow\infty}\f{\mathcal{H}^n(M_i\cap B_2(p_i))}{\mathcal{H}^{n+1}(B_2(p_i))}<\infty.
\endaligned
\end{equation}
Suppose that $M_i$ converges in the induced Hausdorff sense to a metric cone $\mathbf{C}'\subset\mathbf{C}$ with the vertex at $\mathbf{o}$. Namely, there is an $\ep_i$-Gromov-Hausdorff approximation $\Phi_i:\,B_{r_i}(p_i)\to B_{r_i}(\mathbf{o})$ with $\Phi_i(p_i)=\mathbf{o}$ for some sequences $r_i\to\infty$ and $\ep_i\to0$ such that $\Phi_i(M_i\cap B_{r}(p_i))$ converges in the Hausdorff sense to $\mathbf{C}'\cap B_r(\mathbf{o})$ for any $r>0$. Note that $\mathbf{C}'$ may depend on the choice of $\Phi_i$.
We further suppose that $\mathbf{C}'$ splits off $\R^{n-k}$ isometrically and $\mathbf{C}'=CY\times\R^{n-k}$, where $CY\subset CX$ is a metric cone with the vertex at $o$ and the cross section $Y\subset X$.

Suppose that $\La$ is a positive constant such that 
$$\mathcal{H}^n(M_i\cap B_2(p_i))\le\La\mathcal{H}^{n+1}(B_2(p_i))\qquad \mathrm{for\ each}\ i\ge1.$$
Recalling that $B_{R_i}(p_i)$ has Ricci curvature $\ge-nR_i^{-2}$. Then
from Lemma 3.1 in \cite{D1},
\begin{equation*}\aligned
\mathcal{H}^{n+1}\left(B_t(M_i)\cap B_{s}(x_i)\right)\le\f{2t}{1-nR_i^{-1}t}\mathcal{H}^n\left(M_i\cap B_{t+s}(x_i)\right)\le\f{2\La t}{1-nR_i^{-1}t}\mathcal{H}^{n+1}(B_2(p_i))
\endaligned
\end{equation*}
for each $0<t\le\min\{\f{R_i}{n},s\}$ and $s+t<2-d(p_i,x_i)$. With Lemma 4.1 in \cite{D1}, we get
\begin{equation}\aligned
\mathcal{H}^{n+1}\left(B_t(\mathbf{C}')\cap B_{s}(x)\right)\le 2\La t\mathcal{H}^{n+1}\left(B_{2}(o)\right)
\endaligned
\end{equation}
for each $x\in B_2(\mathbf{o})$, each $0<t\le s$ with $s+t\le 2-d(\mathbf{o},\xi)$. Combining the co-area formulas \eqref{coarean+1CXR}\eqref{coareanuX}, there is a constant $\La_n\ge2\La$ depending only on $n,\La$ such that
\begin{equation}\aligned\label{VolBtY}
\mathcal{H}^{k}\left(B_t(Y)\right)\le \La_n t\mathcal{H}^{k}\left(X\right)\qquad \mathrm{for\ any}\ t>0.
\endaligned
\end{equation}

Let $\r_{\mathbf{C}'}=d(\cdot,\mathbf{C}')$ be the distance function from $\mathbf{C}'$ on $\mathbf{C}$, $\r_{CY}=d(\cdot,CY)$ be the distance function from $CY$ on $CX$, $\r_Y=d(\cdot,Y)$ be the distance function from $Y$ on $X$. From \eqref{disXCX} and the diameter of $X$ $\le\pi$, for any point $x\in X$, $t>0$, $z\in\R^{n-k}$
\begin{equation}\aligned\label{rCYrY}
&\r_{\mathbf{C}'}(tx,z)=\inf_{\tau y\in CY,z'\in\R^{n-k}}\sqrt{d_{CY}(tx,\tau y)^2+|z-z'|^2}=\inf_{\tau y\in CY}d_{CY}(tx,\tau y)\\
=&\inf_{\tau y\in CY}\sqrt{t^2+\tau^2-2t\tau\cos d(x,y)}=t\inf_{y\in Y}\sin d(x,y)=t\sin\r_Y(x).
\endaligned
\end{equation}
We will use $k$-dimensional Hausdorff measure as the volume element of integrations on $X$, $(k+1)$-dimensional Hausdorff measure for $CX$, and $(n+1)$-dimensional Hausdorff measure for $\mathbf{C}$. For convenience, we always omit the volume elements if there is no confusions.
\begin{lemma}\label{limsuper}
For any nonnegative Lipschitz function $\phi$ on $X$ with compact support in $X\setminus Y$, we have
\begin{equation}\aligned\label{dsinrYdphi}
\int_X\left\lan d\sin\r_Y,d\phi\right\ran_X\ge k\int_X\phi\sin\r_Y.
\endaligned
\end{equation}
Moreover, for any nonnegative Lipschitz function $\psi$ with compact support in $\mathscr{B}_{\pi/2}(Y)\setminus Y$, we have
\begin{equation}\aligned\label{drYdpsi}
\int_X\left\lan d\r_Y,d\psi\right\ran_X\ge(k-1)\int_X\psi\tan\r_Y.
\endaligned
\end{equation}
\end{lemma}
\begin{proof}
Let $\e$ be a nonnegative Lipschitz function on $\R$ with compact support in $[0,\infty)$.
Let $\z$ be a nonnegative Lipschitz function on $\R^{n-k}$ with compact support and $\sup_{\R^{n-k}}\z>0$.
For any nonnegative Lipschitz function $\phi$ on $X$ with compact support in $X\setminus Y$, we define a function $f$ on $CX$ by
$$f(tx,z)=\phi(x)\e(t)\z(z)\qquad \mathrm{for\ each}\ (tx,z)\in CX\times\R^{n-k}.$$
From \eqref{fzCfz}\eqref{rCYrY}, we have
\begin{equation}\aligned
&\left\lan d\r_{\mathbf{C}'},df\right\ran_\mathbf{C}(tx,z)=\left\lan d(t\sin\r_Y(x)),d(\phi(x)\e(t)\z(z))\right\ran_\mathbf{C}\\
=&\e'(t)\phi(x)\z(z)\sin\r_Y(x)+\f{\e(t)}t\z(z)\left\lan d\sin\r_Y(x),d\phi(x)\right\ran_{X}
\endaligned
\end{equation}
$\mathcal{H}^{n+1}$-a.e. on $\mathbf{C}$. From \eqref{drMdphige0}, we get
\begin{equation}\aligned\label{drXdvphi}
&0\le\int_{\mathbf{C}}\left\lan d\r_{\mathbf{C}'},df\right\ran=\int_{\R^{n-k}}\left(\int_0^\infty t^k\left(\int_{x\in X}\left\lan d\r_{\mathbf{C}'},df\right\ran_{\mathbf{C}}(tx,z)\right)dt\right)dz\\
=&\int_{\R^{n-k}}\z(z) dz\int_0^\infty t^k\e'(t)dt\int_X\phi\sin\r_Y+\int_{\R^{n-k}}\z(z) dz\int_0^\infty t^{k-1}\e(t)dt\int_X\left\lan d\sin\r_Y,d\phi\right\ran_X.
\endaligned
\end{equation}
Since $\e$ has compact support, then
\begin{equation}\aligned
\int_0^\infty t^{k}\e'(t)dt=\int_0^\infty t^{k}d\e(t)=-k\int_0^\infty t^{k-1}\e(t)dt.
\endaligned
\end{equation}
Noting $0<\int_{\R^{n-k}}\z(z) dz<\infty$. Substituting the above equality into \eqref{drXdvphi} gives \eqref{dsinrYdphi}.

Let $\psi$ be a nonnegative Lipschitz function on $X$ with compact support in $\mathscr{B}_{\pi/2}(Y)\setminus Y$. From \eqref{dsinrYdphi} with $\phi=\psi/\cos\r_Y$, we have
\begin{equation}\aligned
&k\int_{\mathscr{B}_{\pi/2}(Y)}\psi\tan\r_Y \le\int_{\mathscr{B}_{\pi/2}(Y)}\left\lan d\sin\r_Y,d(\psi/\cos\r_Y)\right\ran_X \\
=&\int_{\mathscr{B}_{\pi/2}(Y)}\left\lan d\r_Y,d\psi\right\ran_X +\int_{\mathscr{B}_{\pi/2}(Y)}\psi\cos\r_Y\left\lan d\r_Y,d(1/\cos\r_Y)\right\ran_X \\
=&\int_{\mathscr{B}_{\pi/2}(Y)}\left\lan d\r_Y,d\psi\right\ran_X +\int_{\mathscr{B}_{\pi/2}(Y)}\psi\tan\r_Y\left\lan d\r_Y,d\r_Y\right\ran_X ,
\endaligned
\end{equation}
which implies \eqref{drYdpsi} since $\left\lan d\r_Y,d\r_Y\right\ran_X=1$ $\mathcal{H}^k$-a.e. on $X$.
\end{proof}
With Lemma \ref{limsuper}, we have a basic property for $Y$ as follows.
\begin{proposition}\label{VolXpi}
$X\subset \overline{\mathscr{B}_{\pi/2}(Y)}$. Moreover, $\lim_{t\to0}\f1t\mathcal{H}^k(X\setminus\mathscr{B}_{\pi/2-t}(Y))=0$ for $k\ge2$.
\end{proposition}
\begin{proof}
Let $t_1,\cdots,t_4$ be constants satisfying $0<t_1<t_2\le t_3<t_4\le\pi/2$, and $\varphi$ be a Lipschitz function on $\R^+$ with support in $[t_1,t_4]$ satisfying $\varphi(t)=\f{t-t_1}{t_2-t_1}$ on $[t_1,t_2]$, $\varphi(t)=1$ on $[t_2,t_3]$, and $\varphi(t)=\f{t_4-t}{t_4-t_3}$ on $[t_3,t_4]$.
Denote $\psi=\varphi\circ\r_Y$. Then $d\psi=\varphi'd\r_Y$ $\mathcal{H}^k$-a.e. on $X$.
From \eqref{drYdpsi}, we have
\begin{equation}\aligned\label{t1234}
&(k-1)\int_X\psi\tan\r_Y\le\int_X\varphi'\left\lan d\r_Y,d\r_Y\right\ran_X\\
=&\f1{t_2-t_1}\mathcal{H}^k\left(\mathscr{B}_{t_2}(Y)\setminus{\mathscr{B}_{t_1}(Y)}\right)-\f1{t_4-t_3}\mathcal{H}^k\left(\mathscr{B}_{t_4}(Y)\setminus{\mathscr{B}_{t_3}(Y)}\right).
\endaligned
\end{equation}
In particular,
\begin{equation}\aligned\label{t1234*}
\f1{t_4-t_3}\mathcal{H}^k\left(\mathscr{B}_{t_4}(Y)\setminus{\mathscr{B}_{t_3}(Y)}\right)\le\f1{t_2-t_1}\mathcal{H}^k\left(\mathscr{B}_{t_2}(Y)\setminus{\mathscr{B}_{t_1}(Y)}\right).
\endaligned
\end{equation}
Combining \eqref{VolBtY}, it follows that
\begin{equation}\aligned\label{t432}
\f1{t_4-t_3}\mathcal{H}^k\left(\mathscr{B}_{t_4}(Y)\setminus{\mathscr{B}_{t_3}(Y)}\right)\le\f1{t_2}\mathcal{H}^{k}\left(\mathscr{B}_{t_2}(Y)\right)\le \La_n \mathcal{H}^{k}\left(X\right).
\endaligned
\end{equation}

For any fixed constants $0<t<\pi/2<s<\pi$, let $\tilde{\varphi}$ be a Lipschitz function on $\R^+$ with support in $[t,s]$ satisfying $\tilde{\varphi}(r)=\f{r-t}{\pi/2-t}$ on $[t,\pi/2]$, $\tilde{\varphi}(r)=\f{s-r}{s-\pi/2}$ on $[\pi/2,s]$.
We set $\tilde{\psi}=\tilde{\varphi}\circ\r_Y$, then from \eqref{dsinrYdphi}
\begin{equation}\aligned
&k\int_{\mathscr{B}_{s}(Y)\setminus\mathscr{B}_{t}(Y)}\tilde{\psi}\sin\r_Y\le k\int_X\tilde{\psi}\sin\r_Y\le\int_X\tilde{\varphi}'\left\lan d\sin\r_Y,d\r_Y\right\ran_X\\
\le&\f{\cos t}{\pi/2-t}\mathcal{H}^k\left(\mathscr{B}_{\pi/2}(Y)\setminus{\mathscr{B}_{t}(Y)}\right)-\f1{s-\pi/2}\int_{\mathscr{B}_{s}(Y)\setminus{\mathscr{B}_{\pi/2}(Y)}}\cos\r_Y.
\endaligned
\end{equation}
With \eqref{t432}, letting $t\to\pi/2$ in the above inequality implies $\mathcal{H}^k\left(\mathscr{B}_{s}(Y)\setminus{\mathscr{B}_{\pi/2}(Y)}\right)=0$ for any $\pi/2<s<\pi$, which infers
$$X\subset \overline{\mathscr{B}_{\pi/2}(Y)}.$$

For each $0\le t<s\le\pi/2$, we denote $A(t,s)=\mathcal{H}^k\left(\mathscr{B}_{\pi/2-t}(Y)\setminus{\mathscr{B}_{\pi/2-s}(Y)}\right)$. 
For every $t\in(0,\pi/4)$, from \eqref{t1234*} it follows that
\begin{equation}\aligned\label{At1.5tt}
A\left(t,3t/2\right)\ge\f12A\left(0,t\right).
\endaligned
\end{equation}
Let $t_1=\pi/2-2t$, $t_2=t_3=\pi/2-t$, $t_4=\pi/2$. 
Let $\varphi_t$ be a Lipschitz function on $\R^+$ with support in $[t_1,t_4]$ satisfying $\varphi_t(r)=\f{r-t_1}{t}$ on $[t_1,t_2]$, and $\varphi_t(r)=\f{t_4-r}{t}$ on $[t_3,t_4]$.
From \eqref{t1234} and \eqref{At1.5tt}, for $k\ge2$ we have
\begin{equation}\aligned\label{tAt2tcot}
&\f1tA\left(t,2t\right)\ge \f1tA\left(0,t\right)+\int_{\mathscr{B}_{\pi/2}(Y)\setminus{\mathscr{B}_{\pi/2-2t}(Y)}}\varphi_t\circ\r_Y\,\tan\r_Y\\
\ge& \f1tA\left(0,t\right)+\int_{\mathscr{B}_{\pi/2-t}(Y)\setminus{\mathscr{B}_{\pi/2-3t/2}(Y)}}\f{\r_Y-\pi/2+2t}t\tan\r_Y\\
\ge& \f1tA\left(0,t\right)+\f12\cot(3t/2)A\left(t,3t/2\right)\ge\f1tA\left(0,t\right)+\f14\cot(3t/2)A\left(0,t\right).
\endaligned
\end{equation}
There is a constant $t_0>0$ so that $\cot(3t/2)\ge\f1{2t}$ for all $0<t\le t_0$.
From \eqref{tAt2tcot}, we have
\begin{equation}\aligned
A\left(0,2t\right)= A\left(t,2t\right)+A(0,t)\ge A\left(0,t\right)+\f18A\left(0,t\right)+A(0,t)=\f{17}8A(0,t).
\endaligned
\end{equation}
Hence
\begin{equation}\aligned
\f{A(0,t)}t\le\f{16}{17}\f{A\left(0,2t\right)}{2t}\le\cdots\le\left(\f{16}{17}\right)^m\f{A\left(0,2^mt\right)}{2^mt}
\endaligned
\end{equation}
for all $2^mt\le t_0$ with the integer $m\ge1$, which implies $\lim_{t\to0}\f{A(0,t)}t=0$. This completes the proof.
\end{proof}
\begin{remark}
The integer $k$ in \eqref{dsinrYdphi} and the integer $k-1$ in \eqref{drYdpsi} are both sharp. 
Suppose that $X$ is a $k$-dimensional unit sphere $\mathbb{S}^k$ with the standard metric, $Y$ is the equator in $\mathbb{S}^k$, i.e., $Y=\{(x_1,\cdots,x_{k+1})\in\mathbb{S}^k|\, x_{k+1}=0\}$.
Then $\r_Y(x)=\arcsin|x_{k+1}|$ and $\pi/2-\r_Y(x)=\arccos|x_{k+1}|$ for any $x=(x_1,\cdots,x_{n+1})\in \mathbb{S}^k$. 
%Let $\De_X,\De$ denote the Laplacians of $X,\R^{k+1}$, respectively. Note that $X$ has mean curvature $k$ in $\R^{k+1}$. Then
%$$0=\De x_{k+1}=\De_X x_{k+1}+k.$$
Let $p_\pm=(0,\cdots,0,\pm1)\in\mathbb{S}^k$. Then $\pi/2-\r_Y=\r_{p_+}$ on $\mathbb{S}_+^k=\{(x_1,\cdots,x_{k+1})\in\mathbb{S}^k|\, x_{k+1}>0\}$. By Laplacian formula, $\De_X\r_{p_+}=(k-1)\cot\r_{p_+}$ on $\mathbb{S}_+^k\setminus\{p_+\}$, which implies $\De_X\r_Y=-(k-1)\tan\r_Y$ on $\mathscr{B}_{\pi/2}(Y)\setminus Y$. Then
$$\De_X\sin\r_Y=\cos\r_Y\De_X\r_Y-\sin\r_Y=-k\sin\r_Y\qquad \mathrm{on}\ \mathscr{B}_{\pi/2}(Y)\setminus Y.$$
Moreover, $X=\mathscr{B}_{\pi/2}(Y)\cup\{p_+\}\cup\{p_-\}$, and $\mathcal{H}^k(X\setminus\mathscr{B}_{\pi/2-t}(Y))=2\mathcal{H}^k(\mathscr{B}_t(p_+))=2k\omega_k\int_0^t\sin^{k-1} sds$.
Then $\lim_{t\to0}\f1t\mathcal{H}^k(X\setminus\mathscr{B}_{\pi/2-t}(Y))=0$ if and only if $k>1$. Hence the integer $k\ge2$ is sharp in Proposition \ref{VolXpi}.
\end{remark}
Now let us prove a Frankel property on the  metric space $X$ in the following sense.
\begin{theorem}\label{ConnectedY}
For $k\ge2$, $Y$ is connected in $X$.
\end{theorem}
\begin{proof}
Let us prove it by contradiction.
Suppose that there are two non-empty closed sets $Y_1,Y_2\subset Z$ with $Y=Y_1\cup Y_2$ and $Y_1\cap Y_2=\emptyset$.
Let us deduce the contradiction.
There are points $y_1\in Y_1,y_2\in Y_2$ such that
\begin{equation}\aligned\label{defy*z*}
d_X(y_1,y_2)=\inf_{y\in Y_1,y'\in Y_2}d_X(y,y')>0.
\endaligned
\end{equation}
Let $l_{y_1y_2}$ denote the shortest normalized geodesic connecting $y_1,y_2$. Let $w$ be a point in $l_{y_1y_2}$ such that
\begin{equation}\aligned\label{wy*z*}
d_X(w,y_1)=d_X(w,y_2)=\f12d_X(y_1,y_2).
\endaligned
\end{equation}
For any $x\in X\setminus Y$, the triangle inequality implies
\begin{equation}\aligned
&\inf_{y\in Y_1}d_X(x,y)+\inf_{y'\in Y_2}d_X(x,y')\ge\inf_{y\in Y_1,y'\in Y_2}d_X(y,y')= d_X(y_1,y_2)\\
=&d_X(w,y_1)+d_X(w,y_2)\ge\inf_{y\in Y_1}d_X(w,y)+\inf_{y'\in Y_2}d_X(w,y').
\endaligned
\end{equation}
In other words,
\begin{equation}\aligned\label{wY*Z*}
\r_{Y_1}(x)+\r_{Y_2}(x)\ge\r_{Y_1}(w)+\r_{Y_2}(w)\qquad \mathrm{for\ any}\ x\in X\setminus Y.
\endaligned
\end{equation}

From diam$X\le\pi$, let $\th$ be a positive constant denoted by
$$\th=d_X(y_1,y_2)\le\pi.$$
Let us prove $\th<\pi$ by contradiction. Assume $\th=\pi$.
Let $S_t=X\setminus \mathscr{B}_{\pi/2-t}(Y)$. Then $X=\mathscr{B}_{\pi/2-t}(Y_1)\cup \overline{S_t}\cup\mathscr{B}_{\pi/2-t}(Y_2)$.
Let $f_t$ be a Lipschitz function on $X$ defined by 
\begin{eqnarray*}
   f_t= \left\{\begin{array}{ccc}
           -\mathcal{H}^k(\mathscr{B}_{\pi/2-t}(Y_1))  & \quad\quad  {\rm{on}} \ \ \     \mathscr{B}_{\pi/2-t}(Y_2) \\ [3mm]
            \f{\r_{Y_2}-\pi/2}t\mathcal{H}^k(\mathscr{B}_{\pi/2-t}(Y_1))     & \quad\quad\ \ \ \ {\rm{on}} \ \ \  \mathscr{B}_{\pi/2}               (Y_2)\setminus\mathscr{B}_{\pi/2-t}(Y_2) \\ [3mm]
       0     & \quad\quad\ \ \ \ {\rm{on}} \ \ \  X\setminus \mathscr{B}_{\pi/2}(Y) \\ [3mm]
           \f{\pi/2-\r_{Y_1}}t\mathcal{H}^k(\mathscr{B}_{\pi/2-t}(Y_2))  & \quad\quad\ \ \ \ {\rm{on}} \ \ \   \mathscr{B}_{\pi/2}(Y_1)\setminus\mathscr{B}_{\pi/2-t}(Y_1)\\ [3mm]
         \mathcal{H}^k(\mathscr{B}_{\pi/2-t}(Y_2))& \quad\quad\ \ {\rm{on}} \ \ \  \mathscr{B}_{\pi/2-t}(Y_1).
     \end{array}\right.
\end{eqnarray*}
Since $\lim_{t\to0}\f1t\mathcal{H}^k(S_t)=0$ from Proposition \ref{VolXpi}, we have
\begin{equation}\aligned\label{Poinft0}
&\limsup_{t\to0}\left|\int_Xf_t\right|\le\limsup_{t\to0}\int_{S_t}|f_t|\le\mathcal{H}^k(X)\limsup_{t\to0}\mathcal{H}^k(S_t)=0,\\
&\limsup_{t\to0}\int_X|df_t|\le\limsup_{t\to0}\int_{S_t}\f{\mathcal{H}^k(X)}t\le\mathcal{H}^k(X)\limsup_{t\to0}\f{\mathcal{H}^k(S_t)}t=0,
\endaligned
\end{equation}
and
\begin{equation}\aligned\label{Poinft}
\lim_{t\to0}\int_X\left|f_t-\fint_Xf_t\right|=\lim_{t\to0}\int_X|f_t|=2\mathcal{H}^k(\mathscr{B}_{\pi/2}(Y_1))\mathcal{H}^k(\mathscr{B}_{\pi/2}(Y_2)).
\endaligned
\end{equation}
From Poincar\'e inequality \eqref{*Brphi}, the inequalities \eqref{Poinft0}\eqref{Poinft} can not hold simultaneously for the suitable small $t>0$. So we have $\th<\pi$.

Let $\de=\f14\min\{\th/2,1-\sin(\th/2)\}>0$.
By the definition of $M_i$, there are a sequence $\a_i\in(0,\de)$ with $\a_i\to0$, two sequences of closed subsets $M_i',M_i''$ of $M_i\setminus B_{\a_i}(p_i)$ so that $\Phi_i(M_i')$ converges to $CY_1\times\R^{n-k}$ and $\Phi_i(M_i'')$ converges to $CY_2\times\R^{n-k}$ as $i\to\infty$.
Let $\bar{w}=(1w,0^{n-k})\in CX\times\R^{n-k}=\mathbf{C}$, and $w_i\in \p B_1(p_i)$ with $w_i\to \bar{w}$, then
from \eqref{rCYrY} it follows that
\begin{equation}\aligned
B_{\sin(\th/2)+3\de}(w_i)\cap B_{\a_i}(p_i)=\emptyset\qquad \mathrm{for\ the\ suitable\ large}\ i.
\endaligned
\end{equation}
Hence, by the definition of $M_i$ again, both of $M'_i\cap B_{\sin(\th/2)+3\de}(w_i)$ and $M''_i\cap B_{\sin(\th/2)+3\de}(w_i)$ can be written as the supports of rectifiable stationary $n$-varifolds in $B_{\sin(\th/2)+3\de}(w_i)$.
By the definition of $\th$, clearly $M'_i\cap B_{\sin(\th/2)+\de}(w_i)\neq\emptyset$ and $M''_i\cap B_{\sin(\th/2)+\de}(w_i)\neq\emptyset$ for suitable large $i$.
Since $\de\le\th/8$, then $\sin(\th/2)>\th/\pi>2\de$.
Combining Theorem \ref{c1} and \eqref{drMdphige0}, we get
\begin{equation}\aligned
\int_{B_{2\de}(\bar{w})}\left\lan d\r_{CY_1\times\R^{n-k}},d\psi\right\ran \ge0
\endaligned
\end{equation}
for every nonnegative Lipschitz function $\psi$ on $B_{2\de}(\bar{w})$ with compact support in $B_{2\de}(\bar{w})$.
With the proof of Lemma \ref{limsuper}, we deduce
\begin{equation}\aligned\label{BdewrZphitanZ}
\int_{\mathscr{B}_{\de}(w)}\left\lan d\r_{Y_1},d\phi\right\ran_X\ge(k-1)\int_{\mathscr{B}_{\de}(w)}\phi\tan\r_{Y_1}
\endaligned
\end{equation}
for any nonnegative Lipschitz function $\phi$ with compact support in $\mathscr{B}_{\de}(w)$. Analogously, we have
\begin{equation}\aligned\label{BdewrZphitanY*}
\int_{\mathscr{B}_{\de}(w)}\left\lan d\r_{Y_2},d\phi\right\ran_X \ge(k-1)\int_{\mathscr{B}_{\de}(w)}\phi\tan\r_{Y_2}.
\endaligned
\end{equation}
Adding \eqref{BdewrZphitanZ}\eqref{BdewrZphitanY*}, we get
\begin{equation}\aligned\label{rY*+Z*}
\int_{\mathscr{B}_{\de}(w)}\left\lan d(\r_{Y_1}+\r_{Y_2}),d\phi\right\ran_X\ge(k-1)\int_{\mathscr{B}_{\de}(w)}\phi(\tan\r_{Y_1}+\tan\r_{Y_2}).
\endaligned
\end{equation}

Let $\varphi(x)=\r_{Y_1}(x)+\r_{Y_2}(x)-\r_{Y_1}(w)-\r_{Y_2}(w)$ for any $x\in\mathscr{B}_{\de}(w)$, then $\varphi(w)=0$ and $\varphi\ge0$ on $\mathscr{B}_{\de}(w)$ from \eqref{wY*Z*}.
From \eqref{rY*+Z*}, it follows that
\begin{equation}\aligned
\int_{\mathscr{B}_{\de}(w)}\left\lan d\varphi,d\phi\right\ran_X\ge0
\endaligned
\end{equation}
for any nonnegative Lipschitz function $\phi$ with compact support in $\mathscr{B}_{\de}(w)$.
From Lemma \ref{MVIsuper} and $\varphi(w)=0$, we get $\varphi\equiv0$ on $\mathscr{B}_{\de/2}(w)$. However, this contradicts to \eqref{rY*+Z*}. We complete the proof.
\end{proof}

\section{Appendix I}

For each integer $i\ge1$, let $(X_i,d_i)$ be a sequence of compact metric spaces. Let $(X_\infty,d_\infty)$ be a compact metric space such that
there is a sequence of $\ep_i$-Gromov-Hausdorff approximations $\Phi_i:\, X_i\to X_\infty$ for some sequence $\ep_i\to0$.
For each $i$, let $f_i$ be a function on $X_i$. For a function $f$ on $X_\infty$, we say $f_i\rightarrow f$ if $f_i(x_i)\rightarrow f(x)$ for any $x\in X_\infty$ and any sequence $x_i\in X_i$ with $x_i\rightarrow x$.
We further assume that all the $f_i$ are Lipschitz with $\limsup_{i\rightarrow\infty}(\sup_{X_i}|f_i|+\mathbf{Lip} f_i)<\infty$.
\begin{lemma}\label{c00}
There are a subsequence $i'\to\infty$ and a Lipschitz function $f_\infty$ on $X_\infty$ with $\mathbf{Lip} f_\infty\le\limsup_{i\rightarrow\infty}\mathbf{Lip} f_i$ such that $f_{i'}\to f_\infty$.
\end{lemma}
\begin{proof}
Let $\de_k>0$ be a sequence of numbers with $\de_k\rightarrow0$ as $k\rightarrow\infty$. For each $k$, let $\{y_{k,j}\}_{j=1}^{m_k}$ be a finite $\de_k$-net of $X_\infty$.
For each $j\in\{1,\cdots,m_1\}$, let $x_{i,j}\in X_i$ be a sequence converging to $y_{1,j}$ as $i\rightarrow\infty$. Then there is a subsequence $\{1_i\}$ of $\{i\}$ such that $f_{1_i}(x_{1_i,j})$ converges to a number $t_{1,j}\in\R$.
For each $j\in\{1,\cdots,m_2\}$ and for any sequence $x_{1_i,j}\in X_{1_i,j}$ converging as $i\rightarrow\infty$ to $y_{2,j}$, there is a subsequence $\{2_i\}$ of $\{1_i\}$ such that $f_{2_i}(x_{2_i,j})$ converges to a number $t_{2,j}\in\R$.
We continue the procedure, and get a sequence of functions $\{f_{k_i}\}_{i,k\ge1}$ such that $\{(k+1)_i\}$ is a subsequence of $\{k_i\}$,
and for any $y_{k,j}$ and $X_{k_i}\ni x_{k_i,j}\rightarrow y_{k,j}$,
$\lim_{i\rightarrow\infty}f_{k_i}(x_{k_i,j})$ converges to $t_{k,j}\in\R$ as $i\to\infty$.

From $\limsup_{i\rightarrow\infty}\mathbf{Lip} f_i<\infty$, for $y_{k,j}=y_{k',j'}$ there holds $t_{k,j}=t_{k',j'}$.
Now we define a function $f_{\infty}$ on $\cup_{k\ge1}\cup_{1\le j\le m_k}\{y_{k,j}\}$ by letting $f_\infty(y_{k,j})=t_{k,j}$.
From $\limsup_{i\rightarrow\infty}\mathbf{Lip} f_i<\infty$, for any sequence $X_{k_i}\ni x_{k_i,j}'\rightarrow y_{k,j}$ we have 
$$\lim_{i\rightarrow\infty}f_{k_i}(x_{k_i,j}')=\lim_{i\rightarrow\infty}f_{k_i}(x_{k_i,j})=f_{\infty}(y_{k,j}).$$
Moreover, for any $y,y'\in\cup_{k\ge1}\cup_{1\le j\le N_k}\{y_{k,j}\}$ it follows that
$$|f_{\infty}(y)-f_{\infty}(y')|\le\left(\limsup_{i\rightarrow\infty}\mathbf{Lip} f_i\right) d_\infty(y,y').$$
For any $y\in X_\infty$, there is a sequence $\{y_i\}\subset\cup_{k\ge1}\cup_{1\le j\le N_k}\{y_{k,j}\}$ with $y_i\rightarrow y$, and we define $f_\infty(y)=\lim_{i\rightarrow\infty}f_\infty(y_i)$.
From the above inequality, the definition of the function $f_\infty(y)$ is independent of the particular choice of $y_i$. 
Then $\mathbf{Lip} f_\infty\le\limsup_{i\rightarrow\infty}\mathbf{Lip} f_i$, $f_\infty(y)=\lim_{i\to\infty}f_{i_i}(y_i')$ for any $y_i'\in X_{i_i}$ with $y_i'\rightarrow y$.
\end{proof}

For each integer $i\ge1$, let $E_i$ be a closed set in $X_i$. Suppose that $\Phi_i(E_i)$ converges in the Hausdorff sense to a closed set $E_\infty$ in $X_\infty$.
From the triangle inequality, one has
\begin{equation}\aligned
|\r_{E_\infty}(x)-\r_{E_\infty}(y)|\le d_\infty(x,y)\qquad\qquad \mathrm{for\ any}\ x,y\in X_\infty.
\endaligned
\end{equation}
In other words, $\r_{E_\infty}$ has the Lipschitz constant $\mathbf{Lip}\,\r_{E_\infty}\le 1$ on $X_\infty$.

\begin{lemma}\label{c0}
For any point $x\in X_\infty$ and any sequence $x_i\in X_i$ with $x_i\to x$ as $i\to\infty$, we have $\r_{E_i}(x_i)\rightarrow\r_{E_\infty}(x)$.
\end{lemma}
\begin{proof}
The proof is routine. For each $x\in X_\infty$, there is a point $y\in E_\infty$ so that $d_\infty(x,y)=\r_{E_\infty}(x)$.
Let $y_i\in E_i$ with $y_i\rightarrow y$ . Then for any sequence $x_i\in X_i$ with $x_i\rightarrow x$, there holds
\begin{equation}\aligned\label{suprMixi}
\r_{E_\infty}(x)=d_\infty(x,y)=\lim_{i\rightarrow\infty}d_i(x_i,y_i)\ge\limsup_{i\rightarrow\infty}\r_{E_i}(x_i).
\endaligned
\end{equation}
On the other hand, there is a point $z_i\in E_i$ so that $\r_{E_i}(x_i)=d_i(x_i,z_i)$. Suppose there is a sequence $i'\to\infty$ such that $\liminf_{i\rightarrow\infty}\r_{E_i}(x_i)=\lim_{i'\rightarrow\infty}\r_{E_{i'}}(x_{i'})$.
Then there is a subsequence $i''$ of $i'$ such that $z_{i''}\rightarrow z\in E_\infty$.
Hence
\begin{equation}\aligned\label{infrMixi}
\liminf_{i\rightarrow\infty}\r_{E_i}(x_i)=\lim_{i'\rightarrow\infty}d_{i'}(x_{i'},z_{i'})=\lim_{i''\rightarrow\infty}d_{i''}(x_{i''},z_{i''})=d_\infty(x,z)\ge\r_{E_\infty}(x).
\endaligned
\end{equation}
Combining \eqref{suprMixi}\eqref{infrMixi}, we complete the proof.
\end{proof}

\section{Appendix II}

Let $R_i\ge0$ be a sequence with $R_i\rightarrow\infty$ as $i\rightarrow\infty$.
Let $B_{R_i}(p_i)$ be a sequence of $(n+1)$-dimensional smooth geodesic balls with Ricci curvature $\ge-nR_i^{-2}$ such that
$(B_{R_i}(p_i),p_i)$ converges to a metric cone $(\mathbf{C},\mathbf{o})$ in the pointed Gromov-Hausdorff sense.
Suppose $\liminf_{i\to\infty}\mathcal{H}^{n+1}(B_1(p_i))>0$, and for some integer $1\le k\le n$
there is a $k$-dimensional compact metric space $X$ such that $\mathbf{C}=CX\times\R^{n-k}$, where $CX$ is a metric cone with the vertex $o$ of the cross section $X$.

For any point $x\in X$ and $r,t>0$, let $B_r(tx)$ be the ball of radius $r$ and centered at $tx\in CX$ in $CX$,
$\mathscr{B}_r(tx)$ denote the metric ball in $tX=\p B_t(o)$ with the radius $r$ and centered at $tx$, and
$$\mathscr{C}_r(x)\triangleq\{t\xi\in CX|\, \xi\in \mathscr{B}_r(x),\, |t-1|<r/2\}.$$
For $|t-1|<\f r2$ and $0<r\le1$, with Cauchy inequality we get
\begin{equation*}\aligned
&t^2+1-2t\cos(tr)=(t-1)^2+4t\sin^2(tr/2)\ge(t-1)^2+\f4{\pi^2} t^3r^2  r^2\left((t-1)^2+\f{2t^2}{\pi^2}\right)\\
&\ge\f{r^2}{9}\left((t-1)^2+t^2+8(t-1)^2+\f{t^2}8\right)\ge\f{r^2}{9}\left((t-1)^2+t^2+2t(t-1)\right)=\f{r^2}9.
\endaligned
\end{equation*}
Hence from \eqref{disXCX} we have $B_{r/3}(x)\subset\mathscr{C}_r(x)$ for all $r\in(0,1]$. Moreover, for $|t-1|<\f r2$ and $0<r\le1$
\begin{equation}\aligned
(t-1)^2+4t\sin^2(tr/2)\le(t-1)^2+t^3r^2\le\f{r^2}4+\left(1+\f r2\right)^3r^2\le4r^2,
\endaligned
\end{equation}
which implies $\mathscr{C}_r(x)\subset B_{2r}(x)$ for all $r\in(0,1]$. In all, we have
\begin{equation}\aligned\label{Belong}
B_{r/3}(x)\subset\mathscr{C}_r(x)\subset B_{2r}(x)\qquad \mathrm{for\ all}\ r\in(0,1].
\endaligned
\end{equation}
For each $r\in(0,1)$, from the co-area formula \eqref{coareanuX}
\begin{equation}\aligned\label{nuBrxEST}
&\mathcal{H}^{k+1}(\mathscr{C}_r(x))=\int_0^\infty \mathcal{H}^{k}(\p B_t(o)\cap \mathscr{C}_r(x))dt=\int_{1-r/2}^{1+r/2} \mathcal{H}^{k}(t\mathscr{B}_{r}(x))dt\\
=&\int_{1-r/2}^{1+r/2} t^k\mathcal{H}^{k}(\mathscr{B}_{r}(x))dt=\f1{k+1}\left(\left(1+\f r2\right)^{k+1}-\left(1-\f r2\right)^{k+1}\right)\mathcal{H}^{k}(\mathscr{B}_{r}(x)).
\endaligned
\end{equation}
Combining \eqref{Belong}\eqref{nuBrxEST} and Bishop-Gromov volume comparison, there is a constant $c_k\ge1$ depending only on $k$ such that
\begin{equation}\aligned\label{VolDVRr*}
\mathcal{H}^{k}(\mathscr{B}_{R}(x))\le c_k\f{R^k}{r^k}\mathcal{H}^{k}(\mathscr{B}_{r}(x))\qquad \mathrm{for\ each}\ 0<r\le R\le1.
\endaligned
\end{equation}

Let $\phi$ be a Lipschitz function on $X$, and $\la$ be a Lipschitz function on $[0,\infty)$. Put
$$f(t\xi)=\la(t)\phi(\xi)\qquad \mathrm{for\ each}\ t\xi\in CX\ \mathrm{with}\ \xi\in X.$$
From \eqref{DefLip} and \eqref{disXCX}, we have
\begin{equation}\aligned
\mathrm{Lip}\, \phi(\xi)=\limsup_{\e\rightarrow \xi, \e\neq \xi}\f{|\phi(\xi)-\phi(\e)|}{d_{X}(\xi,\e)}.
\endaligned
\end{equation}
and
\begin{equation}\aligned\label{Lipftxi}
(\mathrm{Lip}\, f)^2(t\xi)=\limsup_{\tau\e\rightarrow t\xi, \tau\e\neq t\xi}\f{|f(t\xi)-f(\tau\e)|^2}{d_{CX}(t\xi,\tau\e)^2}=\limsup_{\tau\to t,\e\to\xi, \tau\e\neq t\xi}\f{|\la(t)\phi(\xi)-\la(\tau)\phi(\e)|^2}{t^2+\tau^2-2t\tau\cos d_X(\xi,\e)}.
\endaligned
\end{equation}
Note that
\begin{equation}\aligned
|\la(t)\phi(\xi)-\la(\tau)\phi(\e)|^2=&|\la(t)\phi(\xi)-\la(\tau)\phi(\xi)|^2+|\la(\tau)\phi(\xi)-\la(\tau)\phi(\e)|^2\\
&+2\Big(\la(t)\phi(\xi)-\la(\tau)\phi(\xi)\Big)\Big(\la(\tau)\phi(\xi)-\la(\tau)\phi(\e)\Big).
\endaligned
\end{equation}
Let $\tau-t=\sin\th\sqrt{(\tau-t)^2+t\tau d_X(\xi,\e)^2}$, and $\sqrt{t\tau} d_X(\xi,\e)=\cos\th\sqrt{(\tau-t)^2+t\tau d_X(\xi,\e)^2}$ for $\th\in[0,2\pi)$.
Note that $t^2+\tau^2-2t\tau\cos d_X(\xi,\e)=(t-\tau)^2+t\tau d_X(\xi,\e)^2+t\tau\, O(d_X(\xi,\e)^3)$.
For the fixed $\th$, we have
\begin{equation}\aligned\label{lattauxie}
&\f{|\la(t)-\la(\tau)|^2|\phi|^2(\xi)+|\phi(\xi)-\phi(\e)|^2\la^2(\tau)}{t^2+\tau^2-2t\tau\cos d_X(\xi,\e)}+2\f{(\la(t)-\la(\tau))(\phi(\xi)-\phi(\e))\phi(\xi)\la(\tau)}{t^2+\tau^2-2t\tau\cos d_X(\xi,\e)}\\
\to&(\la'(t))^2\phi^2(\xi)\sin^2\th+\f{\la^2(t)}{t^2}\cos^2\th(\mathrm{Lip}\, \phi)^2(\xi)+2\f{\la(t)\la'(t)}t\phi(\xi)\sin\th\cos\th\mathrm{Lip}\, \phi(\xi)\\
=&\left(\la'(t)\phi(\xi)\sin\th+\f{\la(t)}{t}\cos\th\mathrm{Lip}\, \phi(\xi)\right)^2\qquad\ \  \mathrm{as}\ (\tau-t)^2+t\tau d_X(\xi,\e)^2\to0.
\endaligned
\end{equation}
From \eqref{Lipftxi}-\eqref{lattauxie}, for every $t\xi\in V$ we have
\begin{equation}\aligned\label{Lipftxi**}
&(\mathrm{Lip}\, f)^2(t\xi)=\sup_{\th\in[0,2\pi)}\left(\la'(t)\phi(\xi)\sin\th+\f{\la(t)}{t}\cos\th\mathrm{Lip}\, \phi(\xi)\right)^2\\
=&\f{\la^2(t)}{t^2}(\mathrm{Lip}\, \phi)^2(\xi)+(\la'(t))^2\phi^2(\xi)=\f{\la^2(t)}{t^2}|d\phi|^2(\xi)+(\la'(t))^2\phi^2(\xi),
\endaligned
\end{equation}
where $d\phi$ is the differential of $\phi$ on $X$ defined below \eqref{VolDVRr}.

From the Poincar$\mathrm{\acute{e}}$ inequality \eqref{qPoincare} and the proof of Theorem 1 in \cite{HaK}, it follows that
\begin{equation}\aligned\label{Crf}
\int_{\mathscr{C}_r(x)}\left|f-\fint_{\mathscr{C}_r(x)}f\right|^q\le c_{k}r^q\int_{\mathscr{C}_{r}(x)}|df|^q
\endaligned
\end{equation}
for any $x\in X$, $r\in(0,1]$, $q\ge1$, where $c_{k}$ is a general constant depending only on $k$.
From the co-area formula \eqref{coareanuX}, if $f(t\xi)=\phi(\xi)$ for each $t\xi\in tX$ with $\xi\in X$, then we have
\begin{equation}\aligned\label{CrfBrphi}
&\fint_{\mathscr{C}_r(x)}f=\f1{\mathcal{H}^{k+1}(\mathscr{C}_r(x))}\int_{1-\f r2}^{1+\f r2}\left(\int_{t\mathscr{B}_{r}(x)}f\right)dt
\\
=&\f1{\mathcal{H}^{k+1}(\mathscr{C}_r(x))}\int_{1-\f r2}^{1+\f r2}\left(t^k\int_{\mathscr{B}_{r}(x)}\phi\right) dt=\f1{\mathcal{H}^k(\mathscr{B}_r(x))}\int_{\mathscr{B}_{r}(x)}\phi=\fint_{\mathscr{B}_{r}(x)}\phi.
\endaligned
\end{equation}
With \eqref{Lipftxi**}, substituting \eqref{CrfBrphi} into \eqref{Crf} gives
\begin{equation}\aligned\label{*Brphi*}
\int_{\mathscr{B}_r(x)}\left|\phi-\fint_{\mathscr{B}_r(x)}\phi\right|^q\le c_{k}r^q\int_{\mathscr{B}_{r}(x)}|d\phi|^q
\endaligned
\end{equation}
for any $q\ge1$.

%%%%%%%%%%%%%%%%%%%%%%%%%%%%%%%%%%%%%%%%%%%%%%%%%%%%%%%%%%%%%%%%%%%%%%%%%%%%%%
\bibliographystyle{amsplain}

\end{document}